\documentclass{article}

\usepackage[T1]{fontenc}
\usepackage{mathtools,amsthm,amssymb,enumitem,stmaryrd,fixltx2e,microtype,url}

\setlist[enumerate]{nosep}
\setlist[enumerate,1]{label=\textup{(\alph*)}}

\newtheorem{theorem}{Theorem}[section]
\newtheorem{proposition}[theorem]{Proposition}
\newtheorem{lemma}[theorem]{Lemma}

\theoremstyle{definition}
\newtheorem{definition}[theorem]{Definition}

\theoremstyle{remark}

\newtheorem{remark}[theorem]{Remark}

\numberwithin{equation}{section}

\newcommand{\Lie}[1]{\textsl{#1}}
\newcommand{\lie}[1]{\mathfrak{#1}}
\newcommand{\GL}{\Lie{GL}}
\newcommand{\gl}{\lie{gl}}

\newcommand{\SL}{\Lie{SL}}
\newcommand{\sL}{\lie{sl}}
\newcommand{\SO}{\Lie{SO}}
\newcommand{\so}{\lie{so}}

\newcommand{\Sp}{\Lie{Sp}}
\newcommand{\sP}{\lie{sp}}
\newcommand{\SU}{\Lie{SU}}
\newcommand{\su}{\lie{su}}
\newcommand{\Un}{\Lie{U}}
\newcommand{\un}{\lie{u}}
\newcommand{\Isom}{\Lie{Isom}}
\newcommand{\Isot}{\Lie{Isot}}

\newcommand{\f}{\lie}

\newcommand{\wnabla}{\widetilde{\nabla}}

\newcommand{\e}{\epsilon}
\newcommand{\C}{\mathbb{C}}
\newcommand{\pC}{\widetilde{\mathbb{C}}}
\newcommand{\pCP}{\pC\mathrm P}
\newcommand{\R}{\mathbb{R}}
\newcommand{\bH}{\mathbb{H}}
\newcommand{\pH}{\widetilde{\bH}}
\newcommand{\pHH}{\pH\mathrm H}
\newcommand{\pHP}{\pH\mathrm P}
\newcommand{\V}{\mathcal{V}}

\newcommand{\JJ}{\mathcal J}

\newcommand{\CH}{\C\mathrm H}
\newcommand{\HH}{\bH\mathrm H}

\newcommand{\CP}{\C\mathrm P}
\newcommand{\HP}{\bH\mathrm P}

\DeclareMathOperator{\Ad}{Ad}
\DeclareMathOperator{\ad}{ad}
\DeclareMathOperator{\diag}{diag}

\DeclareMathOperator{\im}{Im}
\DeclareMathOperator{\re}{Re}
\DeclareMathOperator{\Tr}{Tr}
\DeclareMathOperator{\Span}{Span}

\newcommand{\K}{\mathcal{K}}
\newcommand{\sca}{\mathrm{s}}
\newcommand{\QK}{\mathcal{QK}}

\newcommand{\Cyclic}{\mathop{\mathchoice%
  {\vcenter{\hbox{\LARGE\( \mathfrak S \)%
    \vrule width 0pt height 1.7ex depth 0.2ex}}}
  {\vcenter{\hbox{\Large\( \mathfrak S \)\kern-0.1em}}}
  {\vcenter{\hbox{\normalsize\( \mathfrak S \)\kern -0.1em}}}
  {\vcenter{\hbox{\scriptsize\( \mathfrak S \)\kern -0.1em}}}}}

\newcommand{\SXYZ}{XYZ}

\newcommand{\SXJaYJaZ}{XJ_aYJ_aZ}

\DeclarePairedDelimiterX{\inp}[2]{\langle}{\rangle}{#1, #2}
\DeclarePairedDelimiterX{\Set}[2]{\{}{\}}{\, #1 \,\delimsize\vert\, #2 \,}
\DeclarePairedDelimiter{\rcomp}{\llbracket}{\rrbracket}
\DeclarePairedDelimiter{\real}{\lbrack}{\rbrack}

\newcommand{\eqbreak}[1][2]{\\&\hskip#1em}

\makeatletter
\newcommand{\mythanks}{\xdef\@thefnmark{}\@footnotetext}
\makeatother

\title{Non-degenerate homogeneous $\epsilon$-K\"ahler and
$\epsilon$-quaternion K\"ahler structures of linear type}
\author{Ignacio Luj\'an  \and Andrew Swann}
\date{}

\begin{document}

\maketitle

\mythanks{Partially supported by MINECO, Spain, under grant
MTM2011-22528 and by the Danish Council for Independent Research,
Natural Sciences}

\begin{abstract}
  We study the class of non-degenerate homogeneous structures of
  linear type in the pseudo-K\"ahler, para-K\"ahler, pseudo-quaternion
  K\"ahler and para-quaternion K\"ahler cases.  We show that these
  structures characterize spaces of constant holomorphic,
  para-holomorphic, quaternion and para-quaternion sectional curvature
  respectively.  In addition the corresponding homogeneous models are
  computed, exhibiting the relation between these kind of structures
  and the incompleteness of the metric.
\end{abstract}

\mythanks{\emph{MSC2010:} Primary 53C30, Secondary 53C50,
53C55, 53C80.}
\mythanks{\emph{Key words and phrases:} para-K\"ahler,
pseudo-K\"ahler, pseudo-quaternion K\"ahler, para-quaternion
K\"ahler, reductive homogeneous pseudo-Riemannian spaces.}
\renewcommand{\thefootnote}{\arabic{footnote}}

\tableofcontents

\section{Introduction}

Ambrose and Singer \cite{AS} characterized homogeneous Riemannian
spaces by the existence of a \( (1,2) \)-tensor field \( S \) called
\emph{homogeneous structure tensor} (or simply \emph{homogeneous
structure}) satisfying a system of geometric PDE's (called
Ambrose-Singer equations).  This result was extended to homogeneous
spaces with some extra geometric structure in~\cite{Kir}, and later
the theory was adapted to reductive metrics with arbitrary signature
in~\cite{GO2}.  This approach to the study of homogeneous spaces has
proved to be one of the most useful, probably due to the combination
of its algebraic and geometric aspects.  The pointwise classification
of homogeneous structures in the purely Riemannian case was provided
in~\cite{TV}, and in~\cite{Fin}, using representation theory, such a
description was obtained for all the possible holonomy groups in
Berger's list.  These techniques have also been used for metrics with
signature (see for instance~\cite{BGO}).  In many cases (such as
Riemannian, K\"ahler, hyper-K\"ahler, quaternion K\"ahler, as well as
in the pseudo-Riemannian analogues) these classifications contain a
class consisting of sections of a bundle whose rank grows linearly
with the dimension of the manifold.  For that reason homogeneous
structures belonging to these classes are called of \emph{linear
type}.  The corresponding tensor fields are locally characterized by a
set of vector fields satisfying a set of partial differential
equations's determined by the Ambrose-Singer equations.

For definite metrics, in the Riemannian, K\"ahler and quaternion
K\"ahler cases, homogeneous structures of linear type characterize
spaces of negative constant sectional, holomorphic sectional or
quaternionic sectional curvature (see \cite{CGS,GMM,TV}).  When
metrics with signature are studied, the degeneracy of the vector
fields determining the homogeneous structure tensor needs to be taken
into account.  Firstly, the classification of symmetric spaces, that
is spaces with homogeneous tensor zero, becomes much richer (see
\cite{KO}, for example).  Non-zero degenerate homogeneous structures
of linear type, meaning ones that are given by a vector field that is
null, in the purely Riemannian case and in the pseudo-K\"ahler and
para-K\"ahler cases have proved to be related to the geometry of
certain homogeneous plane waves, while in the pseudo-quaternion
K\"ahler and para-quaternion K\"ahler cases the only manifolds
admitting these kind of structures are flat (see
\cite{Mon,CL,CL2}).  On the other hand, non-degenerate structures of
linear type, so the vector field is not null, characterize in the
purely pseudo-Riemannian case spaces of constant sectional curvature
\cite{GO2}.

In this paper we focus on non-degenerate structures of linear type in
the pseudo-K\"ahler, para-K\"ahler, pseudo-quaternion K\"ahler and
para-quaternion K\"ahler cases.  Our main results are that these are
necessarily spaces of constant holomorphic, para-holomorphic,
quaternion and para-quaternion curvature, respectively.  Additionally
we will describe the corresponding homogeneous metrics and show that
in general they are incomplete.

In Section~\ref{sec:preliminaries} the general framework is introduced
and the notation settled.  Throughout the manuscript the notions of
pseudo-K\"ahler and para-K\"ahler geometry and the notions of
pseudo-quaternion and para-quaternion K\"ahler geometry will be
unified and treated together through the definition of \( \e
\)-K\"ahler geometry and \( \e \)-quaternion K\"ahler geometry
respectively.  In Section~\ref{sec:class-homog-struct} we summarize
the pointwise algebraic classification of homogeneous \( \e
\)-K\"ahler and \( \e \)-quaternion K\"ahler structures.  In
Section~\ref{sec:non-degen-homog} we restrict ourselves to the class
of non-degenerate homogeneous \( \e \)-K\"ahler and \( \e
\)-quaternion K\"ahler structures of linear type.  More precisely we
prove that these structures characterize spaces of constant
holomorphic, para-holomorphic, quaternion and para-quaternion
sectional curvature.  In Section~\ref{sec:homog-models-compl} the
associated homogeneous models are computed, showing that they are
usually geodesically incomplete.  This fact will make impossible for
the corresponding simply-connected space forms to admit globally
defined non-degenerate homogeneous structures of linear type.

\section{Preliminaries}
\label{sec:preliminaries}

\subsection{\( \e \)-K\"ahler geometry}

We can gather the notions of pseudo-K\"ahler and para-K\"ahler
geometry in the following way.

\begin{definition}
  Let \( (M,g) \) be a pseudo-Riemannian manifold, and let \(
  \epsilon=\pm 1 \).
  \begin{enumerate}
  \item An \emph{almost \( \epsilon \)-Hermitian structure} on \( (M,g) \) is
    a smooth section \( J \) of \( \so(TM) \) such that \(
    J^2=\epsilon \).
  \item \( (M,g) \) is called \emph{\( \epsilon \)-K\"ahler} if it
    admits a parallel almost \( \epsilon \)-Hermitian structure with
    respect to the Levi-Civita connection.
  \end{enumerate}
\end{definition}

This way \( \e \) should be substituted by \( -1 \) in the
pseudo-K\"ahler case and by \( +1 \) in the para-K\"ahler case.  If \(
(M,g) \) admits an almost \( \e \)-Hermitian structure \( J \), then
\( M \) has dimension \( 2n \) and the signature of \( g \) is \(
(2r,2s) \), \( r+s=n \), for \( \e=-1 \), and \( (n,n) \) for \( \e=1
\).  In addition \( (M,g,J) \) is \( \e \)-K\"ahler if and only if the
holonomy group of \( g \) is contained in \( U(r,s) \) for \( \e=-1 \)
or \( \GL(n,\R)\subset \SO(n,n) \) for \( \e=1 \) (for a survey on
para-complex geometry see \cite{CFG}).

Hereafter \( (M,g,J) \) is supposed to be a connected \( \e
\)-K\"ahler manifold with \( \dim M\geqslant 4 \).  In addition, by
complete we shall mean geodesically complete.

\begin{definition}
  An \( \e \)-K\"ahler manifold \( (M,g,J) \) is called a
  \emph{homogeneous \( \e \)-K\"ahler manifold} if there is a
  connected Lie group \( G \) of isometries acting transitively on \(
  M \) and preserving \( J \).  The structure \( (M,g,J) \) is called
  a \emph{reductive} homogeneous \( \e \)-K\"ahler manifold if the Lie
  algebra \( \f{g} \) of \( G \) can be decomposed as \(
  \f{g}=\f{h}\oplus\f{m} \) with
  \begin{equation*}
    [\f{h},\f{h}]\subset \f{h},\qquad
    [\f{h},\f{m}]\subset \f{m}.
  \end{equation*}
\end{definition}

Using Kiri\v{c}enko's Theorem \cite{Kir} (see also \cite{GO1}) we have

\begin{theorem}
  Let \( (M,g,J) \) be a connected, simply-connected and complete \(
  \e \)-K\"ahler manifold.  Then the following are equivalent:
  \begin{enumerate}
  \item \( (M,g,J) \) is a reductive homogeneous \( \e \)-K\"ahler
    manifold.
  \item \( (M,g,J) \) admits a linear connection \( \wnabla \) such
    that
    \begin{equation}\label{AS+J}
      \wnabla g=0, \quad  \wnabla R=0,\quad  \wnabla
      S=0,\quad  \wnabla J=0,
    \end{equation}
    where \( S=\nabla-\wnabla \), \( \nabla \) is the Levi-Civita
    connection of \( g \), and \( R \) is the curvature tensor of \( g
    \).
  \end{enumerate}
\end{theorem}

\begin{definition}
  A tensor field \( S \) of type \( (1,2) \) satisfying \eqref{AS+J}
  is called a \emph{homogeneous \( \e \)-K\"ahler structure}.
\end{definition}

\subsection{\( \e \)-quaternion K\"ahler geometry}

A pseudo-Riemannian manifold \( (M,g) \) of signature \( (r,s) \) is
called \emph{strongly-oriented} if the bundle of orthonormal frames
reduces to the connected component \( \SO_0(r,s)\subset \SO(r,s) \).
Since \( \SO(r,s)/\SO_0(r,s) \) is discrete there always exists a
strongly oriented cover of \( M \).

Let \( \e=(\e_1,\e_2,\e_3) \) where \( \e_1=-1 \), \(
\e_2=\e_3=\pm 1 \), we can combine the notions of
pseudo-quaternion K\"ahler and para-quaternion K\"ahler geometry
as follows.

\begin{definition}
  \begin{enumerate}
  \item Let \( (M,g) \) be a pseudo-Riemannian manifold.  An \emph{\( \e
    \)-quaternion Hermitian structure} is a subbundle \( Q\subset
    \so(TM) \) with a local basis \( J_a \), \( a=1,2,3 \) satisfying
    \begin{equation*}
      J_a^2=\e_a,\qquad  J_1J_2=J_3.
    \end{equation*}

  \item A pseudo-Riemannian manifold \( (M,g) \) is called \emph{\( \e
    \)-quaternion K\"ahler} if it is strongly-oriented and it admits a
    parallel \( \e \)-quaternion Hermitian structure with respect to
    the Levi-Civita connection.
  \end{enumerate}
\end{definition}

This way \( \e \) should be substituted by \( (-1,-1,-1) \) in the
pseudo-quaternion K\"ahler case and by \( (-1,1,1) \) in the
para-quaternion K\"ahler case.  If \( (M,g) \) admits an \( \e
\)-quaternion Hermitian structure then \( M \) has dimension \( 4n \)
and \( g \) has signature \( (4r,4s) \) for \( \e=(-1,-1,-1) \) and \(
(2n,2n) \) for \( \e=(-1,1,1) \).  In addition \( (M,g) \) is \( \e
\)-quaternion K\"ahler if and only if the holonomy group of the
Levi-Civita connection is contained in \( \Sp(r,s)\Sp(1) \) for \(
\e=(-1,-1,-1) \) and \( \Sp(n,\R)\Sp(1,\R) \) for \( \e=(-1,1,1)
\).  Here \( \Sp(r,s)\Sp(1) \) is seen as a subgroup of \( \SO_0(4r,4s)
\) via the representation \( \left(\C^2\otimes\C^{2n}\right)^{\rho}
\cong \bH^{n} \), where \( \rho \) is the real structure obtained by
multiplication of the quaternionic structures of \( \C^2 \) and \(
\C^{2n} \) respectively, with the quaternionic Hermitian product \(
\inp q{q'} = -\sum_{i=1}^rq_i\overline{q'}_i +
\sum_{j=r+1}^nq_j\overline{q'}_j \).  On the other hand \(
\Sp(n,\R)\Sp(1,\R) \) is seen as a subgroup of \( \SO_0(2n,2n) \) via
the representation \( V=\R^{2}\otimes \R^{2n}\cong \pH^n \), with the
para-quaternion Hermitian product \( \inp q{q'} =
\sum_{i=1}^nq_i\overline{q'}_i \).  We will denote by \( Sp^{\e}(n) \)
the group \( \Sp(r,s) \), \( r+s=n \), when \( \e=(-1,-1,-1) \) and \(
\Sp(n,\R) \) when \( \e=(-1,1,1) \).  Their Lie algebras are denoted by
\( \sP^{\e}(n) \).  For the proof of the following Proposition see
\cite{AC}.

\begin{proposition}
  An \( \e \)-quaternion K\"ahler manifold is Einstein and has
  curvature tensor
  \begin{equation*}
    R=\nu_q R^0+R^{\sP^{\e}(n)},
  \end{equation*}
  where \( \nu_q=\sca/(16n(n+2)) \) is one quarter of the reduced
  scalar curvature, \( R^0 \) is four times the curvature of the \( \e
  \)-quaternionic hyperbolic space (of the corresponding signature)
  \begin{equation}\label{R_0}
    \begin{split}
      R^0_{XYZW} & = g(X,Z)g(Y,W)-g(Y,Z)g(X,W) \eqbreak
      -\sum_{a}\e_a\bigl\{g(J_aX,Z)g(J_aY,W) - g(J_aY,Z)g(J_aX,W)
      \eqbreak[8] +2g(X,J_aY)g(Z,J_aW)\bigr\},
    \end{split}
  \end{equation}
  and \( R^{\sP^{\e}(n)} \) is an algebraic curvature tensor of type
  \( \sP^{\e}(n) \), that is \( R^{\sP^{\e}(n)} \) commutes with \(
  J_a \) for \( a=1,2,3 \).
\end{proposition}

Let \( \{J_a\}_{a=1,2,3} \) be a local basis of \( Q \).  If we define
the associated \( 2 \)-forms \( \omega_a=g(\cdot,J_a\cdot) \), \(
a=1,2,3 \), then it is easy to check that the \( 4 \)-form
\begin{equation*}
  \Omega=\sum_a-\e_a\omega_a\wedge\omega_a
\end{equation*}
is globally defined.  This form is called the \emph{canonical \(
4 \)-form} of \( (M,g,Q) \).

\begin{definition}
  An \( \e \)-quaternion K\"ahler manifold \( (M,g,Q) \) is called a
  \emph{homogeneous \( \e \)-quaternion K\"ahler manifold} if there is
  a connected Lie group \( G \) of isometries acting transitively on
  \( M \) and preserving \( Q \).  The structure \( (M,g,Q) \) is
  called a \emph{reductive} homogeneous \( \e \)-quaternion K\"ahler
  manifold if the Lie algebra \( \f{g} \) of \( G \) can be decomposed
  as \( \f{g}=\f{h}\oplus\f{m} \) with
  \begin{equation*}
    [\f{h},\f{h}]\subset \f{h}, \qquad
    [\f{h},\f{m}]\subset \f{m}.
  \end{equation*}
\end{definition}

Using Kiri\v{c}enko's Theorem  \cite{Kir} we have

\begin{theorem}
  Let \( (M,g,Q) \) be a connected, simply-connected and complete \(
  \e \)-quaternion K\"ahler manifold.  Then the following are
  equivalent:
  \begin{enumerate}
  \item \( (M,g,Q) \) is a reductive homogeneous \( \e \)-quaternion
    K\"ahler manifold. \item \( (M,g,Q) \) admits a linear connection
    \( \wnabla \) such that
    \begin{equation}\label{AS+Omega}
      \wnabla g=0, \quad  \wnabla R=0,\quad  \wnabla
      S=0,\quad  \wnabla \Omega=0,
    \end{equation}
    where \( S=\nabla-\wnabla \), \( \nabla \) is the Levi-Civita
    connection of \( g \), and \( R \) is the curvature tensor of \( g
    \).
  \end{enumerate}
\end{theorem}

\begin{definition}
  A tensor field \( S \) of type \( (1,2) \) satisfying
  \eqref{AS+Omega} is called a \emph{homogeneous \( \e \)-quaternion
  K\"ahler structure}.
\end{definition}

\subsection{Constant sectional curvature}

\begin{definition}
  \begin{enumerate}
  \item An \( \e \)-K\"ahler manifold \( (M,g,J) \) is said to be of
    \emph{constant \( \e \)-holomorphic sectional curvature} \( c \)
    if
    \begin{equation*}
      \begin{split}
        R_{XYZW} = \frac{c}{4}\Bigl\{&g(Y,Z)g(X,W)-g(X,Z)g(Y,W)+\e
        g(X,JZ)g(Y,JW) \eqbreak -\e g(X,JW)g(Y,JZ)+2\e
        g(X,JY)g(Z,JW)\Bigr\}.
      \end{split}
    \end{equation*}
  \item An \( \e \)-quaternion K\"ahler manifold \( (M,g,Q) \) is said
    to be of \emph{constant \( \e \)-quaternion sectional curvature}
    \( c \) if and only if \( R=\frac{c}{4}R^0 \), where \( R^0 \) is
    given by \eqref{R_0}.
  \end{enumerate}
\end{definition}

It is straightforward to adapt the arguments from the well-known case
of definite metrics to prove that two spaces of constant and equal \(
\e \)-holomorphic sectional curvature are locally isometric preserving
their \( \e \)-K\"ahler structures and to prove the corresponding
statement for \( \e \)-quaternion K\"ahler structures.  For a detailed
study of spaces of constant \( \e \)-holomorphic and \( \e
\)-quaternion sectional curvature see for instance \cite{BR,GM,PS,V}.

\begin{proposition}\label{prop:hol-c}
  \begin{enumerate}
  \item Let \( (M,g,J) \) be a \( 2n \)-dimensional connected,
    simply-connected and complete space of constant \( \e
    \)-holomorphic sectional curvature \( c\neq 0 \).  Then, up to
    homothety, it is \( \e \)-holomorphically isometric to \( \CP^n_s
    \) if \( \e=-1 \) and \( c>0 \), \( \CH^n_s \) if \( \e=-1 \) and
    \( c<0 \), or \( \pCP^n \) if \( \e=1 \), where
    \begin{equation}
      \label{complex proj hyperb}
      \CP^n_s = \frac{\SU(n+1-s,s)}{\Lie S(\Un(n-s,s)\times
      \Un(1))},\quad
      \CH^n_s=\frac{\SU(n-s,s+1)}{\Lie S(\Un(n-s,s)\times \Un(1))},
    \end{equation}
    and
    \begin{equation}
      \label{para-complex proj}
      \pCP^n=\frac{\SL(n+1,\R)}{\Lie S(\GL(n,\R)\times \GL(1,\R))}.
    \end{equation}
  \item Let \( (M,g,Q) \) be a \( 4n \)-dimensional connected,
    simply-connected and complete space of constant \( \e
    \)-quaternion sectional curvature \( c\neq 0 \).  Then, up to
    homothety, there is an isometry of \( M \) preserving \( Q \) to
    \( \HP^n_s \) if \( \e=(-1,-1,-1) \) and \( c>0 \), \( \HH^n_s \)
    if \( \e=(-1,-1,-1) \) and \( c<0 \), or \( \pH P^n \) if \(
    \e=(-1,1,1) \), where
    \begin{equation}
      \label{symmet pseudo-quaternion proj hyperb}
      \HP_s^{n}=\frac{\Sp(s,n+1-s)}{\Sp(s,n-s)\Sp(1)},\quad
      \HH_s^{n}=\frac{\Sp(s+1,n-s)}{\Sp(s,n-s)\Sp(1)},
    \end{equation}
    and
    \begin{equation}
      \label{symmet para-quat Kahler}
      \pHP^n=\frac{\Sp(n+1,\R)}{\Sp(n,\R)\Sp(1,\R)}.
    \end{equation}
  \end{enumerate}
\end{proposition}

\begin{remark}\label{remark diffeo CP-CH and HP-HH}
  Note that there is a diffeomorphism between \( \CH_s^{n}(c) \) and \(
  \CP_{n-s}^{n}(-c) \) (for \( c<0 \)) which is an isometry up to a change
  of sign.  Therefore the cases \( c>0 \) and \( c<0 \) are equivalent
  for our purposes, and we can restrict ourselves to one of them.
  The same is true for \( \HH_s^n(c) \) and \( \HP_{n-s}^n(-c) \).
\end{remark}

\section{Classification of homogeneous structures}
\label{sec:class-homog-struct}

\subsection{Homogeneous \( \epsilon \)-K\"ahler structures}

We take the tensor field of type \( (0,3) \) \( S_{XYZ}=g(S_XY,Z)
\).  Hence, equations \( \nabla g=0=\wnabla g \) and \( \nabla J = 0 =
\wnabla J \) imply that \( S\cdot g = 0 \) and \( S\cdot J = 0 \), or
equivalently
\begin{equation*}
  S_{XYZ} = -S_{XZY},\qquad S_{XJYJZ} = -\e S_{XYZ}.
\end{equation*}
For a fixed point \( p\in M \) we denote \( (V,\inp \cdot\cdot) =
(T_pM,g_p) \) and take the space of \( (0,3) \) tensors with the same
symmetries as a homogeneous \( \e \)-K\"ahler structure at~\( p \)
\begin{equation*}
  \K^{\e}(V) \coloneqq \Set{S\in \otimes^3 V^*}{S_{XYZ}=-S_{XZY},\
  S_{XJYJZ}=-\e S_{XYZ}}.
\end{equation*}
In the pseudo-K\"ahler case (\( \e=-1 \)), representation of \(
\Un(p,q) \) gives \cite{BGO} the following decomposition into
irreducible modules:
\begin{equation*}
  \begin{split}
    \K^{-1}(V)
    &= \rcomp{\Lambda^{1,0}} \otimes \real{\Lambda^{1,1}}\\
    &= \rcomp{S^{2,1}_0} \oplus \rcomp{\Lambda^{1,0}} \oplus
    \rcomp{\Lambda^{2,1}_0} \oplus \rcomp{\Lambda^{1,0}}\\
    &= \K^{-1}_1 \oplus \K^{-1}_2 \oplus \K^{-1}_3 \oplus \K^{-1}_4,
  \end{split}
\end{equation*}
where \( V^*\otimes \C = \Lambda^{1,0}\oplus \Lambda^{0,1} \) and \(
\rcomp{\Lambda^{1,0}} = V^* \).  It is easy to prove that if a
homogeneous pseudo-K\"ahler structure belongs to one of the previous
submodules or their direct sum at some point of~\( M \), then the same
is true at every point of~\( M \).  Among these bundles only \(
\K^{-1}_2 \) and \( \K^{-1}_4 \) have rank growing linearly with the
dimension of~\( M \).

On the other hand, in the para-K\"ahler case (\( \e=1 \)), each of the
four real forms splits in to two representations, so \( \K^{1}(V) \)
decomposes into eight irreducible \( GL(n,\R) \)-submodules \(
\K^{1}_1,\dots,\K^{1}_8 \), see~\cite{GO1}.  In this case the
irreducible submodules with dimension growing linearly with the
dimension of \( M \) are labelled \( \K^{1}_2 \), \( \K^{1}_4 \), \(
\K^{1}_6 \), \( \K^{1}_8 \).  We have \( \K^{1}_2\oplus\K^{1}_4 \cong
T^*M \cong \K^{1}_6\oplus\K^{1}_8 \) which has rank \( 2n \) and the
splittings corresponding to the \( \pm 1 \)-eigenspaces of the
para-complex structure.  These eigenspaces have rank~\( n \).

\begin{definition}
  A homogeneous \( \e \)-K\"ahler structure \( S \) is called of
  \emph{linear type} if it belongs pointwise to the submodule
  \begin{enumerate}
  \item \( \K^{-1}_2\oplus\K^{-1}_4 \) for \( \e=-1 \).
  \item \( \K^{1}_2\oplus\K^{1}_4\oplus\K^{1}_6\oplus\K^{1}_8 \) for
    \( \e=1 \).
  \end{enumerate}
\end{definition}

\begin{proposition}[see \cite{GO1}]
  A homogeneous \( \e \)-K\"ahler structure \( S \) is of linear type
  if and only if
  \begin{equation}
    \label{e-Kahler structure}
    \begin{split}
      S_{X}Y &= g(X,Y)\xi - g(\xi,Y)X + \e g(X,JY)J\xi - \e
      g(\xi,JY)JX \eqbreak - 2g(\zeta,JX)JY,
    \end{split}
  \end{equation}
  for some vector fields \( \xi \) and \( \zeta \).
\end{proposition}

\begin{definition}
A homogeneous \( \e \)-K\"ahler structure of linear type \( S \) given by
 formula \eqref{e-Kahler structure} is called (see \cite{BGO})
 \begin{enumerate}
 \item[(i)] \emph{non-degenerate} if \( g(\xi,\xi)\neq 0 \),
 \item[(ii)] \emph{degenerate} if \( g(\xi,\xi)=0 \),
 \item[(iii)] \emph{strongly degenerate} if \( g(\xi,\xi)=0 \) and \(
   \zeta=0 \).
 \end{enumerate}
\end{definition}

It is a straightforward computation to prove (see \cite{BGO})

\begin{proposition}
  A tensor field \( S \) on \( (M,g,J) \) defined by formula
  \eqref{e-Kahler structure} is a homogeneous \( \e \)-K\"ahler
  structure if and only if
  \begin{equation*}
    \wnabla\xi=0,\qquad \wnabla\zeta=0, \qquad \wnabla R=0.
  \end{equation*}
  where \( \wnabla=\nabla-S \).
\end{proposition}

The degenerate and strongly degenerate cases are studied in
\cite{CL2,CL}.  In this paper we will focus on the non-degenerate case.

\subsection{Homogeneous \( \epsilon \)-quaternion K\"ahler structures}

Pointwise homogeneous pseudo-quaternion K\"ahler structures (\(
\e=(-1,-1,-1) \)) were classified in \cite{BGO} (see also \cite{CGS},
and \( \cite{Fin} \) for a representation theoretical approach).  It
is a straightforward task to adapt the techniques used in \cite{CGS}
and \cite{Fin} to the case of homogeneous para-quaternion K\"ahler
structures (\( \e=(-1,1,1) \)).  For that reason, for the sake of
brevity we shall only recall some useful formulas and present the main
results.

Let \( (M,g,Q) \) be an \( \e \)-quaternion K\"ahler manifold of
dimension \( 4n \), \( n\geqslant 2 \).  The property that the holonomy
of the Levi-Civita connection \( \nabla \) is contained in \(
Sp^{\e}(n)Sp^{\e}(1) \) is equivalent to
\begin{equation}
  \label{nabla J}
  \nabla_{X}J_i=\sum_{j=1}^3b_{ij}J_j,\qquad i=1,2,3,
\end{equation}
with \( (b_{ij}) \) a matrix of \( \sP^\e(1) \).  Let \( S \) be a
homogeneous \( \e \)-quaternion K\"ahler structure on \( (M,g,Q) \)
and \( \wnabla=\nabla-S \), the equation \( \wnabla \Omega \) is
equivalent to
\begin{equation}
  \label{wnabla J}
  \wnabla_{X}J_i=\sum_{j=1}^3\widetilde{b}_{ij}J_j,\qquad
  i=1,2,3,
\end{equation}
with \( (\widetilde{b}_{ij}) \) a matrix of \( \sP^\e(1) \).  This
implies that
\begin{equation*}
  J_i(S_XY)-S_X(J_iY)=\sum_{j=1}^3c_{ij}J_jY,\qquad i=1,2,3,
\end{equation*}
with \( (c_{ij}) \) a matrix of \( \sP^\e(1) \).  Note that the matrix
\( (c_{ij}) \) can be obtained taking \( S_X\in
\sP^{\e}(n)+\sP^{\e}(1) \) and projecting to the second
summand.  Taking the \( (0,3) \)-tensor field \( S_{XYZ}=g(S_XY,Z) \),
we have that the symmetries satisfied by a homogeneous \( \e
\)-quaternion K\"ahler structure are
\begin{align}
  S_{XYZ}= & -S_{XZY}\label{symmetries of homog e-quat Kahler struct 1}\\
  S_{XJ_aYJ_aZ}+\e_aS_{X,Y,Z}= &
  \e_b\pi^c(X)g(J_bY,J_aZ)-\e_c\pi^b(X)g(J_cY,J_aZ),\label{symmetries
  of homog e-quat Kahler struct 2}
\end{align}
for any cyclic permutation \( (a,b,c) \) of \( (1,2,3) \), where \(
\pi^1,\pi^2,\pi^3 \) are local \( 1 \)-forms on \( M \), and Einstein
summation convention is used.

We now fix a point \( p\in M \) and denote by \( (V,\inp
\cdot\cdot, J_1,J_2,J_3) \) the tangent space of \( (M,g,Q) \) at
\( p \).  We write \( \V \) for the space of tensors on \( V \)
satisfying formulas \eqref{symmetries of homog e-quat Kahler
struct 1} and \eqref{symmetries of homog e-quat Kahler struct 2}
for some \( \pi^1,\pi^2,\pi^3\in V^* \).  Note that \(
Sp^{\e}(n)Sp^{\e}(1) \) acts in a natural way on \( V \), which
induces a representation of \( Sp^{\e}(n)Sp^{\e}(1) \) on \( \V
\).  We can decompose \( \V \) into irreducible \(
Sp^{\e}(n)Sp^{\e}(1) \)-submodules:
\begin{align*}
  \QK^{\e}_1(V) &  = \Set[\Big]{S \in \V}{S_{XYZ}= \sum_{a=1}^3
  \theta(J_aX)\inp{J_aY}Z,\ \theta\in V^*},\\
  \QK_2^{\e}(V) & =\Set[\Big]{S\in\V}{S_{XYZ}= \sum_{a=1}^3
  \theta^a(X) \inp{J_aY}Z,\ \sum_{a=1}^3 \theta^a \circ
  J_a=0,\ \theta^a \in V^*},\\
  \QK_3^{\e}(V) &  = \Bigl\{S \in \V \Bigm\vert S_{XYZ}= \inp XY
  \theta(Z) - \inp XZ \theta(Y)\eqbreak[5]
  - \sum_{a=1}^3 \e_a\bigl(\inp X{J_aY} \theta(J_aZ)
  - \inp X{J_aZ} \theta(J_aY)\bigr),\,\theta\in V^* \Bigl\},\\
  \QK_4^{\e}(V) &  = \Set[\Big]{S \in \V}{6S_{XYZ}
  = \Cyclic_{\SXYZ} S_{XYZ} - \sum_{a=1}^3
  \e_a\Cyclic_{\quad\mathclap{\SXJaYJaZ}}S_{XJ_aYJ_aZ}, c_{12}(S)=0},
  \\
  \QK_5^{\e}(V) &  = \Set[\Big]{S \in \V}{\Cyclic_{\SXYZ}S_{XYZ}=0},
\end{align*}
where \( c_{12}(S)(Z)=\sum_r\varepsilon^rS_{e_re_rZ} \) for any
orthonormal basis \( \{e_r\} \), and \( \varepsilon^r=\langle
e_r,e_r\rangle \).  It is easy to prove that if a homogeneous \( \e
\)-quaternion K\"ahler structure belongs to one of this submodules at
a point \( p\in M \), then it belongs to the same submodule at every
point.  This means that the previous decomposition define global
classes of homogeneous \( \e \)-quaternion K\"ahler structures, which
we shall denote \( \QK^{\e}_1,\dots,\QK^{\e}_5 \).  Note that \(
\dim(\QK^{\e}_1(V)) = \dim(\QK^{\e}_3(V)) = 4n \) and \(
\dim(\QK^{\e}_2(V)) = 8n \), so that the largest class whose pointwise
submodules have dimension growing linearly with the dimension of the
manifold is \( \QK^{\e}_1 \oplus \QK^{\e}_2 \oplus \QK^{\e}_3 \).

\begin{definition}
  A homogeneous \( \e \)-quaternion K\"ahler structure \( S \) is
  called of \emph{linear type} if it belongs to the class \(
  \QK^{\e}_1 \oplus \QK^{\e}_2 \oplus \QK^{\e}_3 \).  In that case
  \begin{multline}
    \label{struct linear type QT}
    S_XY = g(X,Y)\xi - g(Y,\xi)X -
    \sum_{a=1}^3 \e_a\bigl(g(J_aY,\xi)J_aX-g(X,J_aY)J_a\xi\bigr)\\
    +\sum_{a=1}^3 g(X,\zeta^a)J_aY,
  \end{multline}
  for some local vector fields \( \xi,\zeta^a \), \( a=1,2,3 \).
\end{definition}

\begin{definition}
  A homogeneous \( \e \)-quaternion K\"ahler structure of linear type
  is called
  \begin{enumerate}
  \item \emph{non-degenerate} if \( g(\xi,\xi)\neq 0 \), and
  \item \emph{degenerate} if \( g(\xi,\xi)=0 \).
  \end{enumerate}
\end{definition}

The non-zero degenerate case was studied in \cite{CL2} resulting
that \( (M,g,Q) \) must be flat.  In this paper we shall concentrate in
the non-degenerate case.

\section{Non-degenerate homogeneous structures of linear type}
\label{sec:non-degen-homog}

In this section we show that non-degenerate homogeneous
\( \e \)-K\"ahler and \( \e \)-quaternion K\"ahler structures of linear
type characterizes spaces of constant \( \e \)-holomorphic and
\( \e \)-quaternion sectional curvature respectively.

\subsection{\( \e \)-K\"ahler case}

\begin{lemma}
Let \( (M,g,J) \) be a connected \( \e \)-K\"ahler manifold,
\( \dim M=2n\geqslant 4 \), admitting a non-degenerate homogeneous
\( \e \)-K\"ahler structure of linear type \( S \).  Then \( (M,g,J) \) is
Einstein.
\end{lemma}

\begin{proof}
The following proof has been adapted to the pseudo-Riemannian
setting from one appearing in \cite{GMM} in the Riemannian case.
Equation \( \wnabla R=0 \) reads
\begin{equation}
  \label{nablaR}
  \left(\nabla_XR\right)_{YZWU}=-R_{S_XYZWU}-R_{YS_XZWU}-R_{YZS_XWU}-R_{YZWS_XU},
\end{equation}
so applying the second Bianchi identity and substituting
\eqref{e-Kahler structure} we have
\begin{equation*}
  \begin{split}
    0 & =  \Cyclic_{\SXYZ}\bigl\{2g(X,\xi)R_{YZWU} + g(X,W)R_{YZ\xi U} +
    g(X,U)R_{YZW\xi} \eqbreak + 2\e g(X,JY)R_{J\xi ZWU} + \e
    g(X,JW)R_{YZJ\xi U} + \e g(X,JU)R_{YZWJ\xi}\bigr\}.
  \end{split}
\end{equation*}
Since \( g(\xi,\xi)\neq 0 \) we can choose an orthonormal basis
including  \( \xi/\sqrt{|g(\xi,\xi)|} \).  Contracting the previous
formula with respect to \( X \) and \( W \) and applying first
Bianchi identity we obtain
\begin{equation}
  \label{formula1}
  \begin{split}
    (2n+2)R_{ZY\xi U} & =
    -2g(Y,\xi)r(Z,U) + 2g(Z,\xi)r(Y,U) \eqbreak
    -2\e g(Y,JZ)r(J\xi,U) - g(Y,U)r(Z,\xi) \eqbreak
    -\e g(Y,JU)r(Z,J\xi) + g(Z,U)r(Y,\xi) \eqbreak
    +\e g(Z,JU)r(Y,J\xi),
  \end{split}
\end{equation}
where $r$ is the Ricci curvature.  Denoting  the scalar
curvature by $\sca$, we can deduce \( r(Z,\xi)=(\sca/2n)g(Z,\xi) \) by
contracting \eqref{formula1} with respect to \( Y \) and \( U \)
with the same orthonormal basis as before.  Setting \( a=1/(2n+2)
\) and \( b=\sca/2n \), we can write
\begin{equation}
  \label{formula2}
  \frac{1}{a}R_{\xi U}=2\theta\wedge r(U)-2b\e \theta(JU)F +
  bU^{\flat}\wedge\theta+b(JU)^{\flat}\wedge(\theta\circ J),
\end{equation}
where \( F \) is the symplectic form associated to \( g \) and \( J \).
Using the identity \( R_{WUJ\xi\cdot}=R_{\xi
JWU\cdot}-R_{\xi JUW\cdot} \) we can write \eqref{formula1} as
\begin{equation}\label{formula3}
  \begin{split}
    0 & = 2\theta\wedge R_{WU} + W^{\flat}\wedge R_{\xi
    U}-U^{\flat}\wedge R_{\xi W} \eqbreak
    -2\e F\wedge(R_{\xi JUW}-R_{\xi JWU}) \eqbreak
    -\e (JW)^{\flat}\wedge R_{\xi JU}+\e (JU)^{\flat}\wedge R_{\xi JW}.
  \end{split}
\end{equation}
Denoting  the right hand side of \eqref{formula2} by \( \Xi(U) \) and
substituting in \eqref{formula3} we obtain
\begin{equation*}
  \begin{split}
    0 & = \frac{2}{a}\theta\wedge R_{WU}+W^{\flat}\wedge
    \Xi(U)-U^{\flat}\wedge \Xi(W) \eqbreak
    -2\e F\wedge \left(i_{W}\Xi(JU)-i_{U}\Xi(JW)\right) \eqbreak
    -\e (JW)^{\flat}\wedge \Xi(JU)+\e (JU)^{\flat}\wedge \Xi(JW).
  \end{split}
\end{equation*}
Taking \( W=\xi \) the previous formula transforms into
\begin{equation*}
  0 = \e(2g(\xi,\xi)F+\theta\wedge(\theta\circ J))\wedge(r(JU)-b(JU)^{\flat}),
\end{equation*}
and contracting first with \( \xi \) and then with \( J\xi \) we obtain
\begin{equation*}
  g(\xi,\xi)(r(JU)-b(JU)^{\flat})=0.
\end{equation*}
Since \( g(\xi,\xi)\neq 0 \) we deduce that the manifold is
Einstein.
\end{proof}

\begin{theorem}
  \label{Teorema principal complex}
  Let \( (M,g,J) \) be a connected \( \e \)-K\"ahler manifold, \( \dim
  M=2n\geqslant 4 \), admitting a non-degenerate homogeneous \( \e
  \)-K\"ahler structure \( S \) of linear type.  Then \( (M,g,J) \) has
  constant \( \e \)-holomorphic sectional curvature \( c=-4g(\xi,\xi)
  \) and \( \zeta=0 \).
\end{theorem}

\begin{proof}
  Since by the previous Lemma \( (M,g,J) \) is Einstein, formula
  \eqref{formula1} transforms into
  \begin{equation*}
    R_{YZ\xi W}=cR^{0}_{YZ\xi W},
  \end{equation*}
  where \( c = \sca/(4n(n+1)) \) and \( R^{0} \) is the curvature of
  the \( \e \)-complex hyperbolic space of (real) dimension \( 2n \),
  i.e.,
  \begin{equation*}
    \begin{split}
      R^0_{XYZW}& = g(Y,Z)g(X,W)-g(X,Z)g(Y,W)+\e g(X,JZ)g(Y,JW)
      \eqbreak -\e g(x,JW)g(Y,JZ)+2\e g(X,JY)g(Z,JW).
    \end{split}
  \end{equation*}
  This implies that
  \begin{equation}
    \label{formula4}
    R_{XJX}\xi=c\left\{-2g(JX,\xi)X+2g(X,\xi)JX-2g(X,X)J\xi\right\}.
  \end{equation}
  On the other hand, \( \wnabla \xi=0 \) is equivalent to \(
  \nabla_X\xi = S_X\xi \).  Using this in
  \begin{equation*}
    R_{XJX}\xi=\nabla_{[X,JX]}\xi-\nabla_X\nabla_{JX}\xi+\nabla_{JX}\nabla_X\xi,
  \end{equation*}
  we get
  \begin{equation}
    \label{formula5}
    R_{XJX}\xi=-g(\xi,\xi)R^0_{XJX}\xi+\Theta^{\zeta}_{XJX}\xi,
  \end{equation}
  where
  \begin{equation}
    \begin{split}
      \Theta^{\zeta}_{XY}\xi
      &=2g(X,J\zeta)\left\{g(Y,J\xi)\xi+g(\xi,\xi)JY+2\e
        g(\zeta,Y)J\xi\right\} \\
      & -2g(Y,J\zeta)\left\{g(X,J \xi)\xi+g(\xi,\xi)JX+2\e
        g(X,\zeta)J\xi\right\} \\
      & +2\left\{g(Y,\zeta)g(\xi,JX) - g(X,\zeta)g(\xi,JY) +
        2g(X,JY)g(\xi,\zeta)\right\}J\xi.
    \end{split}
  \end{equation}
  Taking \( Y=X \) and \( X\in\Span\{\zeta,J\zeta\}^{\bot} \), and
  comparing formulas \eqref{formula4} and \eqref{formula5}, we have
  that \( c=-g(\xi,\xi) \) and \( g(\xi,\zeta)=0 \).  In addition, this
  implies that \( \Theta^{\zeta}_{XJX}\xi=0 \), whence \( 2\e
  g(\xi,\xi)g(X,\zeta)=0 \).  This together with \( g(\xi,\zeta)=0 \)
  gives \( \zeta=0 \).

  Let now\( A=R+g(\xi,\xi)R^0 \).  A direct computation from
  \eqref{nablaR} gives
  \begin{equation*}
    \begin{split}
      (\nabla_XR)_{YZWU} & =
      g(Y,\xi)A_{XZWU}+g(Z,\xi)A_{YXWU}+g(W,\xi)A_{YZXU} \eqbreak
      +g(U,\xi)A_{YZWX}-g(JY,\xi)A_{JXZWU}-g(JZ,\xi)A_{YJXWU} \eqbreak
      -g(JW,\xi)A_{YZJXU}-g(JU,\xi)A_{YZWJX}.
    \end{split}
  \end{equation*}
  Since \( A \) satisfies first Bianchi identity, taking cyclic sum in
  \( X,Y,Z \) we obtain
  \begin{equation*}
    0=-2\Cyclic_{\SXYZ}g(X,\xi)A_{YZWU},
  \end{equation*}
  which is equivalent to \( \theta\wedge A_{WU}=0 \).  Contracting with
  \( \xi \) and taking into account that \( A_{YZ\xi W}=0 \) we have
  that
  \begin{equation*}
    0=g(\xi,\xi)A_{WU},
  \end{equation*}
  hence \( A_{WU}=0 \).  This proves that \( (M,g,J) \) has constant \(
  \e \)-holomorphic sectional curvature \( -4g(\xi,\xi) \).
\end{proof}

\begin{remark}
  For \( \e=-1 \), let \( (M,g,J) \) be  connected,
  simply-connected and complete, with a non-degenerate
  homogeneous pseudo-K\"ahler structure of linear type given by the
  vector field~\( \xi \).  If \( g(\xi,\xi)>0 \) then \(
  c=-4g(\xi,\xi)<0 \), so \( M=\C H^n_s \) for some \( s=0,\dots,n-1
  \), with the negative definite case excluded.  Similarly, if \(
  g(\xi,\xi)<0 \) then \( c>0 \), so \( M=\C P^n_s \) for some \(
  s=1,\dots,n \).
\end{remark}

\subsection{\( \e \)-quaternion K\"ahler case}

\begin{theorem}\label{Teorema principal quat}
  Let \( (M,g,Q) \) be a connected \( \epsilon \)-quaternion K\"ahler
  manifold of dimension \( 4n\geqslant 8 \) admitting a non-degenerate
  homogeneous \( \e \)-quaternion K\"ahler structure of linear type \(
  S \).  Then \( S\in\QK_3^{\e} \) and \( (M,g,Q) \) has constant \( \e
  \)-quaternion sectional curvature \( -4g(\xi,\xi) \).
\end{theorem}

\begin{proof}
  We decompose the curvature tensor field of \( (M,g,Q) \) as \(
  R=\nu_qR^0+R^{\sP^{\e}(n)} \) where \( \nu_q=\sca/(16n(n+2)) \) is
  constant (as the manifold is Einstein), \( R^0 \) is \eqref{R_0} and
  \( R^{\sP^{\e}(n)} \) is a curvature tensor field of type \(
  \sP^{\e}(n) \).  Recall that the space of algebraic curvature tensors
  \( \mathcal{R}^{\sP^{\e}(n)} \) is \( [S^4E] \) with \( E=\C^{2n} \)
  for \( \e=(-1,-1,-1) \), and \( S^4E \) with \( E=\R^{2n} \) for \(
  \epsilon=(-1,1,1) \).  Since \( R^0 \) is \(
  Sp^{\epsilon}(n)Sp^{\epsilon}(1) \)-invariant, the covariant
  derivative \( \nabla R^0 \) vanishes.  Moreover, for every vector
  field \( X \), \( S_X \) acts as an element of \(
  \sP^{\e}(n)+\sP^{\e}(1) \), whence \( S R^0=0 \).  Using the second
  equation of \eqref{AS+Omega} and \( \wnabla=\nabla-S \) we have that
  \begin{equation*}
    0=\wnabla R=\nu_q\wnabla R^0+\wnabla R^{\sP^{\e}(n)}=\nabla
    R^{\sP^{\e}(n)}-S R^{\sP^{\e}(n)}.
  \end{equation*}
  Writing \(
  T^*M\otimes(\sP^{\e}(n)+\sP^{\e}(1))=T^*M\otimes\sP^{\e}(n)+T^*M\otimes\sP^{\e}(1)
  \) we can decompose \( S=S_E+S_H \), and hence \( S_H
  R^{\sP^{\e}(n)}=0 \).  We thus obtain
  \begin{equation*}
    \nabla R=\nabla R^{\sP^{\e}(n)}=S_E R^{\sP^{\e}(n)},
  \end{equation*}
  which we can write as
  \begin{equation}
    \label{formula 2 AS}
    (\nabla_XR)_{YZWU} = -R^{\sP^{\e}(n)}_{S_XYZWU}
    -R^{\sP^{\e}(n)}_{YS_XZWU} -R^{\sP^{\e}(n)}_{YZS_XWU} -R^{\sP^{\e}(n)}_{YZWS_XU}.
  \end{equation}
  Taking the cyclic sum in \( X,Y,Z \) and applying Bianchi identities
  we obtain
  \begin{multline*}
    0 =
    \Cyclic_{\SXYZ}\Bigl\{2g(X,\xi)R^{\sP^{\e}(n)}_{YZWU}+g(X,W)R^{\sP^{\e}(n)}_{YZ\xi
      U}+g(X,U)R^{\sP^{\e}(n)}_{YZW\xi}\\
    +2\sum_a\e_a\bigr( g(X,J_aY)R^{\sP^{\e}(n)}_{J_a\xi ZWU}+
      g(X,J_aW)R^{\sP^{\e}(n)}_{YZJ_a\xi U}+
      g(X,J_aU)R^{\sP^{\e}(n)}_{YZWJ_a\xi}\bigr)\Bigr\}.
  \end{multline*}
  Contracting the previous formula with respect to \( X \) and \( W
  \), and taking into account that \( R^{\sP^{\e}(n)} \) is traceless
  we obtain
  \begin{equation*}
    (4n+2)R^{\sP^{\e}(n)}_{YZ\xi U}=0,
  \end{equation*}
  for every vector fields \( Z,Y,U \).  Expanding the expression of \(
  S \) in \eqref{formula 2 AS} and using the previous formula we
  arrive at
  \begin{equation*}
    0=\Cyclic_{\SXYZ}\theta(X)R^{\sP^{\e}(n)}_{YZWU},
  \end{equation*}
  where \( \theta=\xi^{\flat} \), or equivalently
  \begin{equation}\label{wedge 1}0=\theta\wedge
    R^{\sP^{\e}(n)}_{WU}.\end{equation}
  Noting that \( R^{\sP^{\e}(n)} \) satisfies the symmetries \(
  R^{\sP^{\e}(n)}_{XJ_aYWU}+R^{\sP^{\e}(n)}_{J_aXYWU}=0 \), \(
  a=1,2,3 \), we will also have \begin{equation}{\label{wedge
    2}}0=(\theta\circ J_a)\wedge R^{\sP^{\e}(n)}_{WU}=0,\qquad
    a=1,2,3.\end{equation} It is easy to prove that a curvature tensor
  of type \( \sP^{\e}(n) \) satisfying equations \eqref{wedge 1} and
  \eqref{wedge 2} must vanish.  Therefore we conclude that \(
  R=\nu_qR^0 \).

  Now, using the third equation in \eqref{AS+Omega} together with
  \eqref{struct linear type QT}, and taking into account \eqref{wnabla
  J} we have that
  \begin{align*}
    0 &=
    g(X,Y)\wnabla_Z\xi-g(\wnabla_Z\xi,Y)X-\sum_a\epsilon_a(g(\wnabla_Z\xi,J_aY)J_aX+g(X,J_aY)J_a\wnabla_Z\xi)\\
    &
    +\sum_ag(X,\wnabla_Z\zeta^a-\sum_b\widetilde{b}_{ba}(Z)\zeta^b)J_aY.
  \end{align*}
  Let \( \bH^{\e} \) denote the quaternions \( \bH \) or
  para-quaternions \( \pH \) depending on the corresponding value of
  \( \e \).  Taking \( X\in(\bH^{\e}\xi)^{\bot} \) with \( g(X,X)\neq 0
  \), and multiplying by \( X \) in the previous formula we obtain
  that
  \begin{equation}\label{wnabla xi}\wnabla_Z \xi=0.\end{equation}
  Whence
  \begin{equation}\label{wnabla
    zeta}\wnabla_Z\zeta^a=\sum_b\widetilde{b}_{ba}\zeta^b,\qquad
    a=1,2,3.\end{equation} From \eqref{wnabla xi} and \eqref{wnabla J}
  we compute
  \begin{align}\label{nabla Jxi}
    \nabla_XJ_a\xi & =
    \sum_b\widetilde{b}_{ab}(X)J_b\xi+g(X,J_a\xi)\xi\nonumber\\
    & -
    \sum_b\e_b(g(\xi,J_bJ_a\xi)J_bX-g(X,J_bJ_aX)J_b\xi)+\sum_bg(X,\zeta^b)J_bJ_a\xi.
  \end{align}
  On the other hand
  \begin{equation}\label{RXYxi}
    \begin{split}
      R_{XY}\xi & =
      -\nabla_X\nabla_Y\xi+\nabla_Y\nabla_X\xi+\nabla_{[X,Y]}\xi\\
      & = -g(Y,\nabla_X\xi)\xi -g(Y,\xi)\nabla_X\xi
      +g(X,\nabla_Y\xi)\xi
      +g(X,\xi)\nabla_Y\xi\\
      & -\sum_a\e_a\bigl(g(Y,\nabla_XJ_a\xi)J_a\xi +
      g(Y,J_a\xi)\nabla_XJ_a\xi \eqbreak[6] -g(X,\nabla_YJ_a\xi)J_a\xi
      - g(X,J_a\xi)\nabla_YJ_a\xi\bigr) \eqbreak +
      \sum_a-g(Y,\nabla_X\zeta^a)J_a\xi-g(Y,\zeta^a)\nabla_XJ_a\xi
      \eqbreak +g(X,\nabla_Y\zeta^a)J_a\xi+g(X,\zeta^a)\nabla_YJ_a\xi.
    \end{split}
  \end{equation}
  Taking \( X,Y\in(\bH^{\e}\xi)^{\bot} \), we have \( g(R_{XY}\xi,X)=0
  \) from \( R=\nu_qR^0 \) on the one hand, and
  \begin{equation*}
    g(R_{XY}\xi,X)=\sum_ag(X,\zeta^a)g(\xi,\xi)g(J_aY,X)
  \end{equation*}
  from \eqref{RXYxi} on the other.  Moreover, for \( Y=J_bX \) it
  reduces to
  \begin{equation*}
    g(R_{XJ_bX}\xi,X)=-\epsilon_bg(\xi,\xi)g(X,\zeta^b)g(X,X).
  \end{equation*}
  This implies that \( g(X,\zeta^b)=0 \), so that
  \begin{equation*}
    \zeta^b\in\bH^{\e}\xi,\qquad b=1,2,3.
  \end{equation*}
  Recalling \eqref{wnabla xi} we have that \( g(\xi,\nabla_Y\xi)=0
  \).  Applying this and \eqref{nabla J} to \eqref{RXYxi} with \( X=\xi
  \) and \( Y\in(\bH^{\e}\xi)^{\bot} \) we obtain
  \begin{equation*}
    g(Y,\nabla_YJ_a\xi)=0,\qquad g(Y,\nabla_Y\zeta^a)=0,\qquad
    g(\xi,\nabla_Y\zeta^a)=0,
  \end{equation*}
  \begin{equation*}
    g(Y,\nabla_{\xi}J_a\xi)=g(Y,J_a\nabla_{\xi}\xi)+\sum_bg(Y,b_{ab}(\xi)J_b\xi)=0,
  \end{equation*}
  \begin{equation*}
    g(Y,\nabla_YJ_a\xi)=g(\xi,J_a\nabla_Y\xi)+\sum_bg(\xi,b_{ab}J_b\xi)=0.
  \end{equation*}
  Hence
  \begin{align*}
    R_{\xi Y}\xi & =
    g(\xi,\xi)\nabla_Y\xi+\sum_ag(\xi,\zeta^a)\nabla_YJ_a\xi\\
    & =
    -g(\xi,\xi)^2Y-\sum_ag(\xi,\zeta^a)\sum_b\e_bg(J_bJ_a\xi,\xi)J_bY\\
    & = -g(\xi,\xi)^2Y-\sum_ag(\xi,\zeta^a)g(\xi,\xi)J_aY.
  \end{align*}
  Comparing with \( R_{\xi Y}\xi=\nu_q(R^0)_{\xi
  Y}\xi=\nu_qg(\xi,\xi)Y \) we deduce that \( \nu_q=-g(\xi,\xi) \) and
  \( g(\xi,\zeta^a)=0 \).  Finally we take again \(
  X,Y\in(\bH^{\e}\xi)^{\bot} \) in \eqref{RXYxi} obtaining
  \begin{align*}
    R_{XY}\xi & =-g(Y,\nabla_X\xi)\xi+g(X,\nabla_Y\xi)\xi\\
    &
    -\sum_a\e_a(g(Y,\nabla_XJ_a\xi)J_a\xi-g(X,\nabla_YJ_a\xi)J_a\xi)\\
    &+\sum_a-g(Y,\nabla_X\zeta^a)J_a\xi+g(X,\nabla_Y\zeta^a)J_a\xi.
  \end{align*}
  Taking into account \eqref{wnabla zeta}, the previous formula reads
  \begin{equation*}
    R_{XY}\xi  =
    2\sum_a\e_ag(Y,J_aX)g(\xi,\xi)J_a\xi+2\sum_{a,b}\e_bg(Y,J_bX)g(\xi,J_b\zeta^a)J_a\xi,
  \end{equation*}
  and comparing with
  \begin{equation*}
    R_{XY}\xi=\nu_q(R^0)_{XY}\xi=-2\sum_a\e_ag(\xi,\xi)g(X,J_aY)J_a\xi
  \end{equation*}
  we have
  \begin{equation*}
    g(J_b\zeta^a,\xi)=0,\qquad a,b=1,2,3.
  \end{equation*}
  This in conjunction with \( \zeta^a\in(\bH^{\e}\xi)^{\bot} \) and \(
  g(\zeta^a,\xi)=0 \) gives
  \begin{equation*}
    \zeta^a=0,\qquad a=1,2,3.
  \end{equation*}
\end{proof}

\begin{remark}
  For \( \e=(-1,-1,-1) \), let \( (M,g,Q) \) be a connected,
  simply-connected and complete pseudo-quaternion K\"ahler manifold
  admitting a non-degenerate homogeneous pseudo-quaternion K\"ahler
  structure of linear type given by the vector field \( \xi \).  Since
  it has constant pseudo-quaternion sectional curvature \(
  c=-4g(\xi,\xi) \).  If \( g(\xi,\xi)>0 \) then \( c<0 \), so \(
  M=\HH^n_s \) for some \( s=0,\dots,n-1 \); if \( g(\xi,\xi)<0 \)
  then \( c>0 \), so \( M=\HP^n_s \) for some \( s=1,\dots,n
  \).
\end{remark}

\section{Homogeneous models and completeness}
\label{sec:homog-models-compl}

Let \( (M,g,J) \) (respectively \( (M,g,Q) \)) be a connected \( \e
\)-K\"ahler (\( \e \)-quaternion K\"ahler) manifold admitting a
non-degenerate homogeneous \( \e \)-K\"ahler (\( \e \)-quaternion
K\"ahler) structure of linear type \( S \).  By
Proposition~\ref{prop:hol-c} such spaces are now locally isometric to
one of the model spaces \eqref{complex proj
hyperb}--\eqref{para-complex proj} (or \eqref{symmet pseudo-quaternion
proj hyperb}--\eqref{symmet para-quat Kahler}).  We can construct a
Lie algebra \( \f{g} \) using the so called \textit{Nomizu
construction} (see \cite{TV}) in the following way:
\begin{equation*}
  \f{g}=T_pM\oplus \f{hol}^{\wnabla},
\end{equation*}
where \( p\in M \) is a fixed point, \( \wnabla=\nabla-S \), and
the brackets are given by
\begin{equation}
  \left\{
    \begin{aligned}
      \left[A,B\right] & =  AB-BA, &&
        A,B\in \f{hol}^{\widetilde{\nabla}},\\
      \left[A,\eta\right] & =  A\cdot\eta, &&
        A\in \f{hol}^{\widetilde{\nabla}},\eta\in T_pM,\\
      \left[\eta,\zeta\right] & =
        S_{\eta}\zeta-S_{\zeta}\eta+\widetilde{R}_{\eta\zeta}, &&
      \eta,\zeta\in T_pM.
    \end{aligned}
  \right.
\end{equation}
Here \( \widetilde{R} \) stands for the curvature tensor of \( \wnabla \),
that is \( \widetilde{R}=R-R^S \) with the convention
\begin{gather*}
  R_{XY}Z=\nabla_{[X,Y]}Z-\nabla_X\nabla_YZ+\nabla_Y\nabla_XZ,\\
  R^S_{XY}Z=S_{S_XY-S_YX}Z-S_XS_YZ+S_YS_XZ.
\end{gather*}
If the manifold is simply-connected and complete, then it is reductive
homogeneous, and \( \f{g} \) is the Lie algebra of a group acting
transitively on \( M \) and preserving the \( \e \)-K\"ahler or \( \e
\)-quaternion K\"ahler structure.

Using this construction we shall prove the following results.

\begin{theorem}\label{thm:exists}
  The indefinite \( \e \)-complex space forms \( \C P^n_s \), \( \C
  H^n_s \) and \( \pC P^n \) locally admit non-degenerate homogeneous
  \( \e \)-K\"ahler structure of linear type.  Similarly, the
  indefinite \( \e \)-quaternion space forms \( \HP^n_s \), \( \HH^n_s
  \) and \( \pHP^n \) locally admit non-degenerate homogeneous \( \e
  \)-quaternion K\"ahler structure of linear type.
\end{theorem}

\begin{theorem}\label{thm:incomplete}
  Let \( (M,g,J) \) be a connected and simply-connected \( \e
  \)-K\"ahler manifold with \( \dim M\geqslant 4 \), admitting a
  non-degenerate homogeneous \( \e \)-K\"ahler structure of linear
  type.  If \( g \) is not definite then \( (M,g,J) \) is not complete.

  Let \( (M,g,Q) \) be a connected and simply-connected \( \e
  \)-quaternion K\"ahler manifold with \( \dim M\geqslant 8 \),
  admitting a non-degenerate homogeneous \( \e \)-quaternion K\"ahler
  structure of linear type.  If \( g \) is not definite then \(
  (M,g,Q) \) is not complete.
\end{theorem}

\begin{proof}[Procedure for the proof of Theorems \ref{thm:exists}
  and~\ref{thm:incomplete}] This is the general procedure that will be
  specialised to the cases para-K\"ahler, pseudo-K\"ahler,
  para-quateternion K\"ahler, and pseudo-quaternion K\"ahler later.
  Recall the model spaces in Proposition~\ref{prop:hol-c} and write
  each of these as~\( \Isom/\Isot \).  By Remark~\ref{remark diffeo
  CP-CH and HP-HH}, will need only consider the four spaces \( \pCP^n
  \), \( \CH^n_s \), \( \pHH^n \) or \( \HH^n_s \).  By Theorems
  \ref{Teorema principal complex} and~\ref{Teorema principal quat} our
  spaces of linear type are locally \( \pm \)-isometric to one of
  these models.

  The first step is to explicitly compute the Lie algebra
  \( \f{g} = T_pM \oplus \f{hol}^{\wnabla} \) associated with an \( \e
  \)-K\"ahler or \( \e \)-quaternion K\"ahler manifold that is
  homogeneous of linear type with tensor field~\( S \).  This is done by obtaining the
  expression for \( \widetilde{R} = R-R^S \) via Theorems \ref{Teorema
  principal complex} and \ref{Teorema principal quat}.  Now we
  identify \( \f{g} \) with a subalgebra of \( \f{isom} \), in such a
  way that \( \f{hol}^{\wnabla} \) is the intersection of \( \f{g} \)
  and \( \f{isot} \).  This gives subgroups \( G\subset \Isom \) and
  \( H\subset \Isot \), and we find that \( H \) is closed in~\( G \).
  The infinitesimal model \( (\f{g},\f{hol}^{\wnabla}) \) associated
  to~\( S \) is thus regular, so may be realised on the homogeneous
  space \( G/H \).  Now the orbit of \( p = e\Isot \) in the model
  space \( \Isom/\Isot \) is just~\( G/H \).  Counting dimensions one
  sees that \( G/H \) is an open subset of~\( \Isom/\Isot \).  Since
  by construction \( G/H \) admits a non-degenerate homogeneous \( \e
  \)-K\"ahler or \( \e \)-quaternion K\"ahler structure of linear
  type, this would prove Theorem~\ref{thm:exists}.

  Now, \( (M,g) \) is locally isometric to the homogeneous space \(
  G/H \) (see \cite{Tri}) and when \( (M,g) \) is simply-connected and
  complete, so it will be globally isometric to~\( G/H \).  To prove
  Theorem~\ref{thm:incomplete} we show that \( G/H \) is not complete.
  By passing to covers, we may take \( G \) to be simply connected.
  Writing \( \f{h} = \f{hol}^{\wnabla} \), we consider a Lie algebra
  involution \( \sigma\colon \f{g} \to \f{g} \) with \(
  \sigma(\f{h})\subset\f{h} \) and restricting to an isometry for
  the \( \Ad(H) \)-invariant metric on~\( T_pM \).  The map \(
  \sigma \) determines a Lie group involution \( \sigma\colon G\to
  G \) with \( \sigma(H)\subset H \), and an involution~\( \sigma
  \) on the homogeneous space \( G/H \).  Denote the fixed-point set
  of \( \sigma \) on \( X \) by \( X^{\sigma} \).  Then the
  homogeneous spaces \( G^{\sigma}/H^{\sigma} \) and \(
  (G/H)^{\sigma} \) are isometric.  However, \( \sigma \) is an
  isometry, so \( (G/H)^{\sigma} \) is a totally geodesic
  submanifold of~\( G/H \).

  By considering a sequence of such Lie algebra involutions, we can
  construct a chain of totally geodesic submanifolds
  \begin{equation*}
    \dotsb\subset ((G/H)^{\sigma_1})^{\sigma_2}\subset (G/H)^{\sigma_1}\subset G/H.
  \end{equation*}
  In our cases, we use this technique to construct a totally geodesic
  submanifold that we can show is not complete, Lemma~\ref{lemma
  completeness K}.  It follows that \( G/H \) is not geodesically
  complete.
\end{proof}

\begin{lemma}
  \label{lemma completeness K}
  The Lie group \( K \) with Lie algebra \( \f{k}=\Span\{A,V\} \), \(
  [A,V]=V \), and left invariant metric given by
  \begin{equation*}
    g(A,A)=1,\qquad g(V,V)=-1,\qquad g(A,V)=0,
  \end{equation*}
  is not geodesically complete, time-like complete, null complete nor
  space-like complete.
\end{lemma}

\begin{proof}
  The Levi-Civita connection of this metric is
  \begin{equation*}
    \nabla_{A}A=0,\qquad \nabla_{A}V=0,\qquad \nabla_VA=-V,\qquad
    \nabla_VV=-A.
  \end{equation*}
  Let \( \gamma \) be a curve in \( K \) and \( \dot{\gamma} \) its
  derivative.  We write \( \dot{\gamma}(t)=\gamma_1(t)A+\gamma_2(t)V
  \).  The geodesic equation thus implies
  \begin{equation*}
    \left\{
      \begin{aligned}
        &\dot{\gamma}_1-\gamma_2^2=0\\
        &\dot{\gamma}_2-\gamma_1\gamma_2=0.
      \end{aligned}
    \right.
  \end{equation*}
  The solution to this system with space-like initial value \(
  \gamma_1(0) = 0 \), \( \gamma_2(0) = 1 \) is \( \gamma_1(t) =
  \tan(t) \), \( \gamma_2 = 1/\cos(t) \) which is defined for \(
  -\pi/2<t<\pi/2 \).  On the other hand, the null initial value \(
  \gamma_1(0) = 1 = \gamma_2(0) \), has solution \( \gamma_1(t) =
  \gamma_2(t) = 1/(1-t) \) which is only defined for \( t<1 \).
  Finally, the time-like initial value \( \gamma_1(0) = 1 \), \(
  \gamma_2(0) = r \), \( 0<r<1 \), has \( x(t) = s \coth(st+k) \), \(
  y(t) = s /\sinh(st+k) \), where \( s = \sqrt{1-r^2} \), \( \tanh k =
  s \).  These solutions are only defined for \( t \ne -k/s \).
\end{proof}

We now prove Theorems \ref{thm:exists} and~\ref{thm:incomplete} for
the para-K\"ahler, pseudo-K\"ahler, para-quaternion K\"ahler and
pseudo-quaternion K\"ahler cases.  Due to differences, we treat them
separately.

\subsection{Para-K\"ahler case}\label{subsection
para-Kahler case}

During this subsection \( \pC \) denotes the set of para-complex
numbers, \( e \) stands for the imaginary
para-complex unit, so \( e^2=+1 \), and \( \bar{z} \) denotes the
para-complex conjugation of \( z\in \pC \).

We first compute the infinitesimal model \( (\f{g},\f{hol}^{\wnabla})
\).  Using formula \eqref{e-Kahler structure} with \( \zeta=0 \) and \(
\e=1 \), we obtain by direct calculation
\begin{equation*}
  \begin{split}
    R^S_{XY}Z & = g(\xi,\xi)\left\{g(Y,Z)X-g(X,Z)Y+ g(Y,JZ)JX-
      g(X,JZ)JY\right\} \eqbreak
     -2 g(X,JY)\left\{g(\xi,JZ)\xi+g(\xi,Z)J\xi\right\},
  \end{split}
\end{equation*}
and since \( (M,g,J) \) has constant para-holomorphic sectional
curvature we have
\begin{equation*}
  \widetilde{R}_{XY}Z=-2
  g(X,JY)\left\{g(\xi,\xi)JZ-g(\xi,JZ)\xi-g(\xi,Z)J\xi\right\}.
\end{equation*}
Now, \( \widetilde{R}_{XY}\xi=0 \) and thus \( \widetilde{R}_{XY}
\) acts trivially on \( \R^2=\Span\{\xi,J\xi\} \).  On the other
hand for \( Z\in\Span\{\xi,J\xi\}^{\bot} \), one has
\begin{equation*}
  \widetilde{R}_{XY}Z=-2g(X,JY)g(\xi,\xi)JZ,
\end{equation*}
so that \( \widetilde{R}_{XY} \) acts on \( U=\Span\{\xi,J\xi\}^{\bot}
\) as \( -2 g(X,JY)g(\xi,\xi)J \).  We conclude that \(
\f{hol}^{\wnabla} \) is one dimensional and is generated by the
element \( \JJ=\frac{1}{2g(\xi,\xi)^2}\widetilde{R}_{\xi J\xi} \).
The remaining brackets are
\begin{equation}\label{corchetes paraC}
  \begin{aligned}
    [Z_1,Z_2]&=2 g(Z_1,JZ_2)L_0, & [\xi,J\xi]&=2g(\xi,\xi)L_0,\\
    [\xi,Z]&=g(\xi,\xi)JZ & [J\xi,Z]&=g(\xi,\xi)JZ,
  \end{aligned}
\end{equation}
where \( Z_1,Z_2,Z\in U \) and \( L_0=J\xi-g(\xi,\xi)\JJ \).
The Lie algebra given by the Nomizu
construction is thus
\begin{equation*}
  \f{g}=\R\JJ\oplus \Span\{\xi,J\xi\}\oplus U.
\end{equation*}
On the other hand, the description \eqref{para-complex proj} of \( \pC P^n
\) as a symmetric space has Cartan decomposition
\begin{equation*}
  \sL(n+1,\R)=\f{s}(\gl(n,\R)\oplus\gl(1,\R)) \oplus \f{m} \subset \so(n+1,n+1),
\end{equation*}
with
\begin{equation*}
  \f{m}=\Set*{
  \begin{pmatrix}
    0_n & v\\
    -v^* & 0
  \end{pmatrix}
  }{v\in\pC^{n}}.
\end{equation*}
We write \( \pC^n = \R^n + e\R^n \).  The algebra \( \sL(n+1,\R)
\) decomposes as
\begin{equation*}
  \sL(n+1,\R) = \f{s}\left(\gl(n,\R)\oplus\gl(1,\R)\right) \oplus \f{a}
  \oplus \f{n}_1 \oplus \f{n}_2,
\end{equation*}
where
\begin{equation*}
  \f{a}=\R A_0,\qquad
  A_0 =
  \begin{pmatrix}
    0_{n-1} & 0 & 0\\
    0  & 0 & e\\
    0  & e & 0
  \end{pmatrix},
\end{equation*}
is a maximal \( \R \)-diagonalisable subalgebra of \( \f{m} \), and
\begin{equation*}
  \f{n}_1=\Set*{
  \begin{pmatrix}
    0_{n-1} & -ev & v\\
    -ev^* & 0 & 0\\
    -v^*& 0 & 0
  \end{pmatrix}
  }{v\in\pC^{n-1}},
\quad
  \f{n}_2=\Set*{
  \begin{pmatrix}
    0_{n-1} & 0 & 0\\
    0 &  -eb & b\\
    0 & -b & eb
  \end{pmatrix}
  }{b\in\R},
\end{equation*}
are the eigenspaces of the positive restricted roots \(
\Sigma^+=\{\lambda,2\lambda\} \) with \( \lambda(A_0)=1 \).

We shall identify \( \f{g} \) with a subalgebra of \( \sL(n+1,\R) \)
following arguments analogous to those in~\cite{CGS1}.  First it is
obvious that \( \JJ\in \f{s}\left(\gl(n,\R)\oplus\gl(1,\R)\right) \),
and since \( \JJ \) acts trivially on \( \Span\{\xi,J\xi\} \) and
effectively on \( U \), the space \( U \) can be identified with \(
\f{n}_1 \) and \( \Span\{\xi,J\xi\}\subset \R\JJ+\f{a}+\f{n}_2 \).
Now, from \eqref{corchetes paraC} it easily follows that \(
L_0\in\f{n}_2 \), and since \( \xi \) has only real eigenvalues on \(
\f{g} \), we can take \( \xi=g(\xi,\xi)A_0 \) up to a Lie algebra
automorphism.  Let
\begin{equation*}
  X =
  \begin{pmatrix}
    0_{n-1} & 0 & 0\\
    0  & -e & 1\\
    0  & -1 & e
  \end{pmatrix}.
\end{equation*}
Using a Lie algebra automorphism we can take \( L_0=X \) which gives
\( J\xi=X+g(\xi,\xi)\JJ \).  Finally, identifying \( U \) with \(
\f{n}_1 \) and \( \f{n}_1 \) with \( \pC^{n-1} \) in the obvious way,
we have from \eqref{corchetes paraC} \( [v,w]=2g(v,Jw)X \).

From the matrix expression of \( \f{n}_1 \) we obtain \(
[v,w]=-2\inp v{ew}X \), where \( \inp vw =\re\sum_j\bar{v}_jw_j
\), \( v,w\in U\equiv\f{n}_1\equiv\pC^{n-1} \).  Comparing this
two expressions we conclude that \( J \) is acting on \( U \) as
multiplication by \( -e \), therefore \( \JJ \) must be
\begin{equation*}
  \JJ=\frac{e}{n+1}\diag((-2)^{n-1},(n-1)^2),
\end{equation*}
with powers denoting multiplicities.

Regarding the Lie algebra involutions involved in the proof of
Theorem~\ref{thm:incomplete} we take \( \sigma\colon \f{g}  \to \f{g}
\) given by
\begin{gather*}
  \JJ \mapsto -\JJ,\qquad A_0 \mapsto A_0,\qquad X+g(\xi,\xi)\JJ
  \mapsto -\left(X+g(\xi,\xi)\JJ\right),\\
  v \mapsto -\overline{v},\quad v\in\f{n}_1\equiv\pC^{n-1}
\end{gather*}
and \( \tau\colon \f{g}^{\sigma}  \to  \f{g}^{\sigma} \) with
\begin{equation*}
  A_0  \mapsto   A_0,\qquad
  (v_1,\dots,v_{n-2},v_{n-1})^T \mapsto  (-v_1,\dots,-v_{n-2},v_{n-1})^T.
\end{equation*}
We thus have
\begin{equation*}
  \f{k}=(\f{g}^{\sigma})^{\tau}=\Set*{
  \begin{pmatrix}
    0_{n-2}  & 0 & 0 & 0\\
    0 & 0 & 0 & es\\
    0 & 0 & 0 & et\\
    0 & -es & et & 0
  \end{pmatrix}
}{s,t\in\R},
\end{equation*}
and the chain of totally geodesic submanifolds
\begin{equation*}
  K=(G^{\sigma})^{\tau} \subset G^{\sigma}=(G/H)^{\sigma} \subset G/H,
\end{equation*}
where \( K \) is as in Lemma~\ref{lemma completeness K}, and is
incomplete.

\subsection{Pseudo-K\"ahler case}
\label{subsection pseudo-Kahler case}

During this subsection \( i \) denotes the imaginary complex unit.
The computations of the infinitesimal model \(
(\f{g},\f{hol}^{\nabla}) \) are completely analogous to those in the
previous subsection setting \( \e=-1 \).  We obtain that
\begin{equation*}
  \widetilde{R}_{XY}Z=2 g(X,JY)g(\xi,\xi)JZ,
\end{equation*}
so that \( \f{hol}^{\wnabla} \) is the one dimensional Lie algebra
generated by \( \JJ=\frac{1}{2g(\xi,\xi)^2}\widetilde{R}_{\xi J\xi}
\).  The remaining brackets are
\begin{equation}
  \label{corchetes pseudoC}
  \begin{aligned}
    [Z_1,Z_2]&=-2g(Z_1,JZ_2)L_0, & [\xi,J\xi]&=2g(\xi,\xi)L_0,\\
    \left[\xi,Z\right]&=g(\xi,\xi)JZ & [J\xi,Z]&=g(\xi,\xi)JZ,
  \end{aligned}
\end{equation}
where \( Z_1,Z_2,Z\in U \) and \( L_0=J\xi-g(\xi,\xi)\JJ \).
Nomizu's construction gives the Lie algebra
\begin{equation*}
  \f{g}=\R\JJ\oplus \Span\{\xi,J\xi\}\oplus U,
\end{equation*}
where \( U=\Span\{\xi,J\xi\}^{\bot} \).  On the other hand, recall
description \eqref{complex proj hyperb} of \( \CH^n_s \) as symmetric
space.  The Riemannian case \( \CH^n_0 \) is studied in~\cite{CGS1}.
We then suppose \( s>0 \), and for the sake of simplicity we also
suppose \( 2s<n-1 \), the opposite case is analogous.  Let
\begin{equation*}
  \varepsilon=\begin{pmatrix}0 & 1\\ 1& 0\end{pmatrix}
  \quad\text{and}\quad
  \Sigma=\diag\bigl((1)^{n-2s-1},(\varepsilon)^{s+1}\bigr).
\end{equation*}
We have
\begin{equation*}
  \su(n-s,s+1)=\Set{C\in\gl(n+1,\C)}{C^*\Sigma+\Sigma C=0,\
  \Tr(C)=0},
\end{equation*}
so that \( \su(n-s,s+1) \) decomposes as
\begin{equation*}
  \su(n-s,s+1) = \f{s}\left(\un(n-s,s)\oplus\un(1)\right) \oplus \f{a}
  \oplus \f{n}_1 \oplus \f{n}_2,
\end{equation*}
where \( \f{a}=RA_0 \), \( A_0=\diag(0,\dots,0,1,-1) \),
\begin{equation*}
  \f{n}_1=\Set*{
  \begin{pmatrix}
    0_{n-1} & 0 & v\\
    -\left(\Sigma'v\right)^* & 0 & 0\\
    0 & 0 & 0
  \end{pmatrix}
  }{v\in\C^{n-1}},
  \quad
  \f{n}_2=\Set*{
  \begin{pmatrix}
    0_{n-1} & 0 & 0\\
    0 & 0 & ib\\
    0 & 0 & 0
  \end{pmatrix}
  }{b\in\R},
\end{equation*}
for \( \Sigma'=\diag\bigl((1)^{n-2s-1},(\varepsilon)^s\bigr) \).  As
in the para-K\"ahler case we identify \( \f{g} \) with a subalgebra of
\( \su(n-s,s+1) \).  More precisely we have that \( U \) is identified
with \( \f{n}_1 \), \( \xi=g(\xi,\xi)A_0 \), and \( J\xi =
L_0+g(\xi,\xi)A_0 \) with
\begin{equation*}
  L_0=
  \begin{pmatrix}
    0_{n-1} & 0 & 0\\
    0 & 0 & i\\
    0 & 0 & 0
  \end{pmatrix}.
\end{equation*}
In addition, from the matrix representation of \( \f{n}_1 \) we obtain
\begin{equation*}
  \JJ=\frac{i}{n+1}\diag\bigl( (-2)^{n-1},(n-1)^2 \bigr).
\end{equation*}

Regarding the Lie algebra involutions involved in the proof of
Theorem~\ref{thm:incomplete} we take \( \sigma\colon \f{g} \to \f{g}
\) defined by
\begin{gather*}
  \JJ \mapsto -\JJ,\qquad A_0 \mapsto A_0,\qquad X+g(\xi,\xi)\JJ
  \mapsto -\bigl(X+g(\xi,\xi)\JJ\bigr),\\
  v \mapsto \overline{v},\quad v\in\f{n}_1\equiv\C^{n-1}\\
\end{gather*}
and \( \tau\colon \f{g}^{\sigma}  \to  \f{g}^{\sigma} \) with
\begin{equation*}
  A_0  \mapsto   A_0,
    \qquad (v_1,\dots,v_{n-2},v_{n-1})^T
    \mapsto (-v_1,\dots,-v_{n-1},-v_{n-2})^T.
\end{equation*}
Then
\begin{equation*}
  \f{k}=(\f{g}^{\sigma})^{\tau}=\Set*{
  \begin{pmatrix}
    0_{n-3} & 0 & 0 & 0 & 0\\
    0 & 0 & 0 & 0 & t\\
    0 & 0 & 0 & 0 & -t\\
    0 & t & -t & s & 0\\
    0 & 0 & 0 & 0 & -s
  \end{pmatrix}
  }{s,t\in\R},
\end{equation*}
and we have the following chain of totally geodesic submanifolds:
\begin{equation*}
  K=(G^{\sigma})^{\tau}\subset  G^{\sigma}=(G/H)^{\sigma} \subset
  G/H,
\end{equation*}
where \( K \) is as in Lemma \ref{lemma completeness K}.

\subsection{Para-quaternion K\"ahler case}

For this section $\pH$ denotes the set of para-quaternions with
imaginary units $i,j,k$.  Using~\eqref{struct linear type QT} we
compute
\begin{equation*}
  \begin{split}
    R^S_{XY}W &=
    -g(\xi,\xi)\Bigl\{g(X,W)Y-g(Y,W)X \eqbreak[8]
    +\sum_a\epsilon_a(g(X,J_aW)J_aY-g(Y,J_aW)J_aX)\Bigr\} \eqbreak
    -2\sum_a\epsilon_a(g(\xi,J_aW)g(X,J_aY)\xi+g(\xi,W)g(X,J_aY)J_a\xi) \eqbreak
    +2\sum_a(g(X,J_aY)g(\xi,J_cW)J_b\xi-g(X,J_aY)g(\xi,J_bW)J_c\xi),
  \end{split}
\end{equation*}
where \( (a,b,c) \) is a cyclic permutation of \( (1,2,3) \), and
\( (\epsilon_1,\epsilon_2,\epsilon_3)=(-1,1,1) \).  From \(
\widetilde{R}=R-R^S \) and \( R=-g(\xi,\xi)R^0 \), we obtain
\begin{equation*}
  \begin{split}
    \widetilde{R}_{XY}W & =
    - 2 \sum_a \epsilon_a g(\xi,\xi) g(X,J_aY) J_aW \eqbreak
    + 2 \sum_a g(X,J_aY)\bigl(\epsilon_a g(\xi,J_aW) \xi + \epsilon_a
    g(\xi,W) J_a \xi - g(\xi,J_cW) J_b\xi \eqbreak[12] + g(\xi,J_bW)J_c\xi\bigr).
  \end{split}
\end{equation*}
In particular
\begin{gather*}
  \widetilde{R}_{XY}\xi =0,\\
  \begin{aligned}
    \widetilde{R}_{XY}J_1\xi & =
    4g(\xi,\xi)\bigl(g(X,J_2Y)J_3\xi-g(X,J_3Y)J_2\xi\bigr),\\
    \widetilde{R}_{XY}J_2\xi & =
    4g(\xi,\xi)\bigl(g(X,J_1Y)J_3\xi-g(X,J_3Y)J_1\xi\bigr),\\
    \widetilde{R}_{XY}J_3\xi & =
    4g(\xi,\xi)\bigl(g(X,J_2Y)J_1\xi-g(X,J_1Y)J_2\xi\bigr),
  \end{aligned}\\
  \widetilde{R}_{XY}Z
  =-2g(\xi,\xi)\sum_a\epsilon_ag(X,J_aY)J_aZ,\quad
  \text{for \( Z\in(\widetilde{\bH}\xi)^{\bot} \).}
\end{gather*}
This shows that \( \f{hol}^{\wnabla} \) acts on
\begin{equation*}
  T_pM = \R\xi+\im\pH \xi+(\widetilde{\bH}\xi)^{\bot}
\end{equation*}
as \( \sP(1,\R) \) acts on the representation
\begin{equation*}
  \R+\sP(1,\R)+EH,
\end{equation*}
where \( E=\R^{2n-2} \) and \( H=\R^2 \).  In addition, for \(
Y\in(\widetilde{\bH}X)^{\bot} \) we have \( \widetilde{R}_{XY}=0 \),
and for \( X \) such that \( g(X,X)=1/(2g(\xi,\xi)) \) we have
\begin{gather*}
  \widetilde{R}_{XJ_aX}\xi=0,\qquad
  \widetilde{R}_{XJ_aX}J_b\xi=-[J_a,J_b]\xi\qquad \text{and}\\
  \widetilde{R}_{XJ_aX}Z=-J_aZ,\quad
  Z\in(\widetilde{\bH}\xi)^{\bot}.
\end{gather*}
Denoting by \( \JJ_a \) the element of \( \f{hol}^{\wnabla} \)
that acts as \( J_a \) on \( (\widetilde{\bH}\xi)^{\bot} \), the
remaining brackets of \( \f{g} \) are
\begin{align}
  [Z_1,Z_2] &=2\sum_a\epsilon_ag(Z_1,J_aZ_2)(J_a\xi-g(\xi,\xi)\JJ_a),
  \label{corchete inicial paracomp}\\
  [\xi,Z] & = g(\xi,\xi)Z,\\
  [J_a\xi,Z]&= g(\xi,\xi)J_aZ,\\
  [\xi,J_a\xi]&= 2g(\xi,\xi)J_a\xi-2g(\xi,\xi)^2\JJ_a,\\
  [J_a\xi,J_b\xi]&=\epsilon_c\left(4g(\xi,\xi)J_c\xi-2g(\xi,\xi)^2\JJ_c\right),
  \label{corchete final paracomp}
\end{align}
for \( (a,b,c) \) any cyclic permutation of \( (1,2,3) \), where \(
Z,Z_1,Z_2\in (\pH X)^{\bot} \).  The Nomizu construction thus gives us
the Lie algebra
\begin{equation*}
  \f{g} = T_pM+\f{hol}^{\wnabla} \cong
  \R\xi+\im\pH \xi+(\pH\xi)^{\bot}+\sP(1,\R),
\end{equation*}
where \( \f{hol}^{\wnabla} \) acts on \( T_pM \) as \( \sP(1,\R)
\) acts on the representation \( \R+\sP(1,\R)+\pH^{n-1} \).  We now
identify this algebra with a subalgebra of \( \sP(n+1,\R) \).  The
algebra \( \sP(n+1,\R)=\Set{A\in\gl(n+1,\pH)}{A+A^*=0} \) has
Cartan decomposition
\begin{equation*}
  \sP(n+1,\R)=\sP(n,\R)+\sP(1,\R)+\f{p},
\end{equation*}
where
\begin{gather*}
  \sP(n,\R)+\sP(1,\R)=\Set*{
  \begin{pmatrix}
    A & 0 \\
    0 & q
  \end{pmatrix}
  }{A\in\sP(n,\R),q\in\im\pH},
  \\
  \f{p}=\Set*{
  \begin{pmatrix}
    0 & v\\
    -v^* & 0
  \end{pmatrix}
  }{v\in\pH^n}.
\end{gather*}
The maximal abelian subalgebra of \( \f{p} \) is up to isomorphism
\( \f{a}=\Span\{A_0\} \), where
\begin{equation*}
  A_0=
  \begin{pmatrix}
    0_{n-1} & 0 & 0\\
    0 & 0 & j\\
    0 & j & 0
  \end{pmatrix}.
\end{equation*}
The restricted roots are \( \{\pm\lambda,\pm 2\lambda\} \), where
\( \lambda(A_0)=1 \).  With the choice of positive roots
\( \{\lambda,2\lambda\} \), the corresponding root spaces are
\begin{gather*}
  \f{n}_1=\Set*{
  \begin{pmatrix}
    0_{n-1} & -vj & v\\
    -j\bar{v} & 0 & 0\\
    -\bar{v} & 0 & 0
  \end{pmatrix}
  }{v\in\pH^{n-1}}
  ,\\
  \f{n}_2=\Set*{
  \begin{pmatrix}
    0_{n-1} & 0 & 0\\
    0 & -jqj & jq\\
    0 & \bar{q}j & q
  \end{pmatrix}
  }{q\in\im\pH}.
\end{gather*}
Therefore, the algebra \( \sP(n+1,\R) \) decomposes as
\begin{equation*}
\sP(n+1,\R)=\sP(n,\R)+\sP(1,\R)+\f{a}+\f{n}_1+\f{n}_2.
\end{equation*}
By a similar argument to~\cite[\S 5.2]{CGS}, one can identify \(
\f{hol}^{\wnabla} \) with the second summand in \( \sP(n,\R)+sp(1,\R)
\).  We consider the \( \ad \)-invariant complement \(
\f{m}_{\lambda}=\f{a}+\f{n}_1+\f{p}_{\lambda} \) where
\begin{equation*}
  \f{p}_{\lambda}=\Set*{
  \begin{pmatrix}
    0_{n-1} & 0 & 0\\
    0 & (\lambda-1)jqj & jq\\
    0 & \bar{q}j & (\lambda+1)q
  \end{pmatrix}
  }{q\in\im\pH}
\end{equation*}
and \( \lambda \in\R \).  From the brackets \eqref{corchete inicial
paracomp}--\eqref{corchete final paracomp} we see that \(
\xi\in\f{a} \) and \( J_a\xi\in\f{p}_{\lambda} \), and  by the
holonomy action we identify \( \f{n}_1 \) with \( (\pH\xi)^{\bot}
\).  In addition, comparing the brackets
\begin{equation*}
[Z_1,Z_2] =-2\sum_a\epsilon_ag(J_aZ_1,Z_2)(J_a\xi-g(\xi,\xi)\JJ_a),
\end{equation*}
and
\begin{equation*}
  \begin{split}
    [v,w]&= 2 \inp{v(-i)}w
    \begin{pmatrix}
      0_{n-1} & 0 & 0\\
      0 & i & -k\\
      0 & -k & i
    \end{pmatrix}
    - 2\inp{v(-j)}w
    \begin{pmatrix}
      0_{n-1} & 0 & 0\\
      0 & -j & 1\\
      0 & -1 & j
    \end{pmatrix}
    \eqbreak
    - 2\inp{v(-k)}w
    \begin{pmatrix}
      0_{n-1} & 0 & 0\\
      0 & k & -i\\
      0 & -i & k
    \end{pmatrix}
    ,
  \end{split}
\end{equation*}
where \( v,w\in\f{n}_1\cong (\pH\xi)^{\bot}\cong\pH^{n-1} \) (as
para-quaternion vector spaces) and \( \inp vw=\re(v^*w) \), we
have
\begin{gather*}
  J_1\xi - g(\xi,\xi)\JJ_1 =
  \begin{pmatrix}
    0_{n-1} & 0 & 0\\
    0 & i & -k\\
    0 & -k & i
  \end{pmatrix}
  ,\quad
  J_2\xi - g(\xi,\xi)\JJ_2=
  \begin{pmatrix}
    0_{n-1} & 0 & 0\\
    0 & -j & 1\\
    0 & -1 & j
  \end{pmatrix},
\\
J_3\xi - g(\xi,\xi)\JJ_3 =
\begin{pmatrix}
  0_{n-1} & 0 & 0\\
  0 & k & -i\\
  0 & -i & k
\end{pmatrix}.
\end{gather*}
Hence \( \JJ_1 \) acts on \( \f{n}_1 \) as right multiplication by
\( -i \), etc., that is
\begin{equation*}
  \JJ_1 =
  \begin{pmatrix}
    0_{n-1} & 0 & 0\\
    0 & -i & 0\\
    0 & 0 & i
  \end{pmatrix}
  \quad  \JJ_2=
  \begin{pmatrix}
    0_{n-1} & 0 & 0\\
    0 & j & 0\\
    0 & 0 & j
  \end{pmatrix},
  \quad
  \JJ_3 =
  \begin{pmatrix}
    0 & 0 & 0\\
    0 & -k & 0\\
    0 & 0 & k
  \end{pmatrix}.
\end{equation*}

Regarding the Lie algebra involutions involved in the proof of
Theorem~\ref{thm:incomplete} we take \( \sigma\colon \f{g} \to \f{g}
\) given by
\begin{gather*}
  \JJ_1 \mapsto -\JJ_1,\quad \JJ_2 \mapsto \JJ_2,\quad
  \JJ_3  \mapsto  -\JJ_3,\\
  \xi \mapsto \xi,\quad J_1\xi \mapsto -J_1\xi,\quad J_2\xi \mapsto
  J_2\xi,\quad
  J_3\xi  \mapsto  -J_3\xi,\\
  v_1+iv_2+jv_3+kv_4 \mapsto v_1-iv_2+jv_3-kv_4,
\end{gather*}
for \( v_1+iv_2+jv_3+kv_4\in(\pH\xi)^{\bot} \).  We then let \(
\tau\colon \f{g}^{\sigma}  \to  \f{g}^{\sigma} \) be
\begin{equation*}
  \JJ_2  \mapsto  -\JJ_2,\quad
  \xi  \mapsto  \xi,\quad
  J_2\xi  \mapsto  -J_2\xi,\quad
  v_1+iv_2  \mapsto
  -v_1+jv_2,
\end{equation*}
and additionally define \(   \lambda\colon (\f{g}^{\sigma})^{\tau}
\to  (\f{g}^{\sigma})^{\tau} \) by
\begin{equation*}
  \xi  \mapsto  \xi,\qquad
  (v_1j,\dots,v_{n-2}j,v_{n-1})^T
  \mapsto  (-v_1j,\dots,-v_{n-2}j,+v_{n-1}j)^T,
\end{equation*}
The fixed point set of the sequence \( \sigma \), \( \tau \), \( \lambda \) is
\begin{equation*}
  \f{k}=\Span\{\xi,(0,\dots,0,j)\},
\end{equation*}
so that the chain of totally geodesic submanifolds is
\begin{equation*}
  K\subset (G^{\sigma})^{\tau}=(G^{\sigma}/H^{\sigma})^{\tau}\subset
  G^{\sigma}/H^{\sigma}=(G/H)^{\sigma} \subset G/H.
\end{equation*}
Once again it is easy to see that \( K \) is as in Lemma \ref{lemma
completeness K}.

\subsection{Pseudo-quaternion K\"ahler case}

Throughout this section $i,j,k$ are the imaginary units of the
quaternions~\( \bH \). With the help of formula \eqref{struct linear
type QT} we compute
\begin{align*}
  R^S_{XY}W & =-g(\xi,\xi)\Bigl\{g(X,W)Y-g(Y,W)X\eqbreak[6]
  +\sum_a(-g(X,J_aW)J_aY+g(Y,J_aW)J_aX)\Bigr\}\\
 &
 -2\sum_a\left\{g(\xi,J_aW)g(X,J_aY)\xi+g(\xi,W)g(X,J_aY)J_a\xi\right\}\\
 &+2\sum_a\left\{g(X,J_aY)g(\xi,J_cW)J_b\xi-g(X,J_aY)g(\xi,J_bW)J_c\xi\right\},
\end{align*}
where \( (a,b,c) \) is a cyclic permutation of \( (1,2,3) \).  Then
\( \widetilde{R}=R-R^S \) gives
\begin{align*}
\widetilde{R}_{XY}W & =-2g(\xi,\xi)\sum_ag(J_aX,Y)J_aW\\
& +2\sum_a\left\{g(\xi,J_aW)g(X,J_aY)\xi+g(\xi,W)g(X,J_aY)J_a\xi\right\}\\
 &-2\sum_a\left\{g(X,J_aY)g(\xi,J_cW)J_b\xi-g(X,J_aY)g(\xi,J_bW)J_c\xi\right\}.
\end{align*}
In particular
\begin{gather*}
  \widetilde{R}_{XY}\xi =0,\\
  \begin{aligned}
    \widetilde{R}_{XY}J_1\xi
    &=-4g(\xi,\xi)\left\{g(J_3X,Y)J_2\xi-g(J_2X,Y)J_3\xi\right\},\\
    \widetilde{R}_{XY}J_2\xi
    &=-4g(\xi,\xi) \left\{g(J_1X,Y)J_3\xi-g(J_3X,Y)J_1\xi\right\},\\
    \widetilde{R}_{XY}J_3\xi &
    =-4g(\xi,\xi)\left\{g(J_2X,Y)J_1\xi-g(J_1X,Y)J_2\xi\right\},
  \end{aligned}
  \\
  \widetilde{R}_{XY}Z=-2g(\xi,\xi)\sum_ag(J_aX,Y)J_aZ,\qquad \text{for
  \( Z\in(\bH\xi)^{\bot} \)}.
\end{gather*}
This implies that \( \f{hol}^{\wnabla} \) acts over \( T_pM \) as
\( \sP(1) \) in the representation
\begin{equation*}
  T_pM=\R\xi+\im\bH\xi+(\bH\xi)^{\bot} =\R+\sP(1)+[EH],
\end{equation*}
where here \( E=\C^{n-1} \).  In addition, for \( Y\in(\bH
\xi)^{\bot} \) we have \( \widetilde{R}_{XY}=0 \), and for \( X \)
such that \( g(X,X)=1/(2g(\xi,\xi)) \) we have
\begin{equation*}
  \widetilde{R}_{XJ_aX}\xi  =0,\quad
  \widetilde{R}_{XJ_aX}J_bX =-[J_a,J_b]\xi,\quad
  \widetilde{R}_{XJ_aX}Z = -J_aZ.
\end{equation*}
We denote by \( \JJ_a \) the element \( -\widetilde{R}_{XJ_aX} \),
which acts as \( J_a \) on the factor \( [EH]\cong
(\bH\xi)^{\bot} \).  The remaining brackets of \( \f{g} \) are
given by
\begin{align} [Z_1,Z_2]
  &=2\sum_a\left\{g(J_aZ_1,Z_2)J_a\xi-g(\xi,\xi)g(J_aZ_1,Z_2)\JJ_a\right\}
  \label{corchete inicial},\\
  [\xi,Z] &= g(\xi,\xi)Z,\\
  [\xi,J_a\xi] &=2g(\xi,\xi)J_a\xi-2g(\xi,\xi)^2\JJ_a,\\
  [J_a\xi,Z] &=g(\xi,\xi)J_aZ,\\
  [J_a\xi,J_b\xi]
  &=4g(\xi,\xi)J_c\xi-2g(\xi,\xi)^2\JJ_c,\label{corchete final}
\end{align}
for \( Z,Z_1,Z_2\in(\bH\xi)^{\bot} \) and each cyclic permutation
\( (a,b,c) \) of \( (1,2,3) \).  The Lie algebra
produced by Nomizu's construction is thus
\begin{equation*}
  \f{g}=T_pM+\f{hol}^{\wnabla}=\R\xi+\im\bH\xi+(\bH\xi)^{\bot}+\sP(1),
\end{equation*}
where \( \f{hol}^{\wnabla} \) acts on \( T_pM \) as \( \sP(1) \)
on \( \R+\sP(1)+\bH^{n-1} \).  Recalling description \eqref{symmet
pseudo-quaternion proj hyperb} of \( \HH^{n}_s \) as a symmetric
space, we identify \( \f{g} \) with a subalgebra of \(
\sP(n-s,s+1) \).  The Riemannian case \( \HH^n_0 \) is studied in
\cite{CGS}, for that reason we suppose \( s>0 \).  We can also
suppose \( n-2s-1>0 \) for the sake of simplicity.  Let
\begin{equation*}
  \varepsilon=
  \begin{pmatrix}
    0 & 1\\
    1 & 0
  \end{pmatrix}
  ,\quad
  \Sigma=\diag\bigl((1)^{n-2s-1},(\varepsilon)^{s+1}\bigr),
\end{equation*}
we have that
\begin{equation*}
  \sP(n-s,s+1)=\Set{A\in\gl(n+1,\bH)}{A^*\Sigma+\Sigma A=0}.
\end{equation*}
The algebra \( \f{sp(n-s,s+1)} \) decomposes as
\begin{equation*}
  \sP(n-s,s+1)=\sP(n-s,s)+\sP(1)+\f{a}+\f{n}_1+\f{n}_2,
\end{equation*}
where \( \f{a} \) is generated by \( A_0=\diag(0,\dots,0,1,-1) \) and
\begin{equation*}
  \f{n}_1=\Set*{
  \begin{pmatrix}
    0_{n-1} & 0 & v\\
    -(\Sigma'v)^* & 0 & 0\\
    0 & 0 & 0
  \end{pmatrix}
  }{v\in\bH^{n-1}},\quad
  \f{n}_2=\Set*{
  \begin{pmatrix}
    0_{n-1} & 0 & 0\\
    0 & 0 & b\\
    0 & 0 & 0
  \end{pmatrix}
  }{b\in\im\bH},
\end{equation*}
with \( \Sigma'=\diag\bigl((1)^{n-2s-1},(\varepsilon)^s) \).
Following the same arguments as in the para-quaternion K\"ahler
case, \( \f{hol}^{\wnabla} \) is identified with the second
summand in the Lie subalgebra \( \sP(n-s,s)+\sP(1) \).  We also
identify \( \f{n}_1 \) with \( (\bH\xi)^{\bot} \), and from the matrix
expression of \( \f{n}_1 \) we obtain
\begin{equation*}
  J_1\xi-g(\xi,\xi)\JJ_1=
  \begin{pmatrix}
    0_{n-1}&0&0\\
    0 & 0 & i\\
    0 & 0 & 0
  \end{pmatrix},\quad\text{etc.}
\end{equation*}
In addition we have
\begin{equation*}
  \JJ_1=
  \begin{pmatrix}
    0&0&0\\
    0 & i & 0\\
    0 & 0 & i
  \end{pmatrix}
  ,\quad\JJ_2=
  \begin{pmatrix}
    0&0&0\\
    0 & j & 0\\
    0 & 0 & j
  \end{pmatrix}
  ,\quad\JJ_3=
  \begin{pmatrix}
    0&0&0\\
    0 & k & 0\\
    0 & 0 & k
  \end{pmatrix}.
\end{equation*}

For the Lie algebra involutions involved in the proof of
Theorem~\ref{thm:incomplete} we finally take \(   \sigma\colon \f{g}
\to  \f{g} \) given by
\begin{gather*}
        \JJ_1  \mapsto  \JJ_1,\quad
        \JJ_2  \mapsto  -\JJ_2,\quad
        \JJ_3  \mapsto  -\JJ_3\\
        \xi  \mapsto  \xi,\quad
        J_1\xi  \mapsto  J_1\xi,\quad
        J_2\xi  \mapsto  -J_2\xi,\quad
        J_3\xi  \mapsto  -J_3\xi\\
        v_1+iv_2+jv_3+kv_4  \mapsto
       v_1+iv_2-jv_3-kv_4,
\end{gather*}
for \( v_1+iv_2+jv_3+kv_4\in(\bH\xi)^{\bot} \).  Then we put \(
\tau\colon \f{g}^{\sigma} \to \f{g}^{\sigma} \) to be
\begin{equation*}
        \JJ_1  \mapsto  -\JJ_1,\quad
        \xi  \mapsto  \xi,\quad
        J_1\xi  \mapsto  -J_1\xi,\quad
        v_1+iv_2  \mapsto
       v_1-iv_2,
\end{equation*}
and define \( \lambda\colon (\f{g}^{\sigma})^{\tau} \to
(\f{g}^{\sigma})^{\tau} \) by
\begin{equation*}
  \xi  \mapsto  \xi,\qquad
  (v_1,\dots,v_{n-2},v_{n-1})^T
  \mapsto  (-v_1,\dots,-v_{n-1},-v_{n-2})^T.
\end{equation*}
This leads to the chain of totally geodesic submanifolds
\begin{equation*}
  K=((G^{\sigma})^{\tau})^{\lambda}\subset
  (G^{\sigma})^{\tau}=(G^{\sigma}/H^{\sigma})^{\tau}\subset
  G^{\sigma}/H^{\sigma}=(G/H)^{\sigma} \subset G/H,
\end{equation*}
with \( K \) as in Lemma \ref{lemma completeness K}, and so incomplete.

\bigskip
{\small
  \setlength{\parindent}{0pt}
  Ignacio Luj\'an

  Departamento de Geometr\'\i a y Topolog\'\i a, Facultad de Matem\'aticas,
  Universidad Complutense de Madrid, Av.\ Complutense s/n,
  28040--Madrid, Spain.

  \textit{E-mail}: \url{ilujan@mat.ucm.es}

  \smallskip Andrew Swann

  Department of Mathematics, Aarhus University, Ny Munkegade 118, Bldg
  1530, DK-8000 Aarhus C, Denmark \textit{and}
  CP\textsuperscript3-Origins, Centre of Excellence for Cosmology and
  Particle Physics Phenomenology, University of Southern Denmark,
  Campusvej 55, DK-5230 Odense M, Denmark.

  \textit{E-mail}: \url{swann@imf.au.dk}
  \par}


\end{document}